\documentclass[11pt, oneside, reqno]{amsart}   	
\usepackage{geometry}                		
\usepackage{amsmath}
\usepackage{amsthm}
\usepackage{xspace}
\usepackage[psamsfonts]{amssymb}
\usepackage[latin1]{inputenc}
\usepackage{graphicx,color}
\usepackage[curve]{xypic} 
\usepackage{pifont}
\usepackage{pdfpages}
\usepackage{ textcomp }
\usepackage[normalem]{ulem}
\usepackage{marvosym}
\usepackage{tikz-cd}
\usepackage{hyperref}
\usepackage{changepage}
\usepackage{pgfplots}
\usepackage{hyperref}
\usepackage{fancyhdr}
\usepackage{color}
\usepackage{amssymb}
\usepackage{amsfonts}
\usepackage{ marvosym }
\usepackage{mdframed}
\usepackage[T1]{fontenc}
\usepackage{pifont}
\usepackage[width=.75\textwidth]{caption}
\usepackage{array}
\usepackage{algorithm}
\usepackage{algpseudocode}
\usepackage{setspace}
\usepackage{tikz-3dplot}
\usepackage{dsfont}

\newtheorem*{thm*}{Theorem}
\newtheorem{thm}{Theorem}[section]

\newtheorem{cor}[thm]{Corollary}
\newtheorem{prop}[thm]{Proposition}
\newtheorem{lem}[thm]{Lemma}

\theoremstyle{definition}
\newtheorem{defn}[thm]{Definition}

\newtheorem{exmp}[thm]{Example}

\theoremstyle{remark}
\newtheorem{rem}[thm]{Remark}

\numberwithin{figure}{section}
\numberwithin{equation}{subsection}

\algrenewcommand\algorithmicrequire{\textbf{Input:}}
\algrenewcommand\algorithmicensure{\textbf{Output:}}

\usepackage{style}
\newcommand{\x}{\mathbf{x}}

\DeclareMathOperator*{\argmax}{argmax}

\begin{document}

\thispagestyle{empty}

\title{On Finding the Closest Zonotope to a Polytope in Hausdorff Distance}
\author{George D. Torres}

\maketitle
\makeatletter
\makeatother

\begin{abstract}
We provide a local theory for the optimization of the Hausdorff distance between a polytope and a zonotope. To do this, we compute explicit local formulae for the Hausdorff function $d(P, -) : \ZZ_n \to \R$, where $P$ is a fixed polytope and $\ZZ_n$ is the space of rank $n$ zonotopes. This local theory is then used to provide an optimization algorithm based on subgradient descent that converges to critical points of $d(P, -)$. We also express the condition of being at a local minimum as a polyhedral feasibility condition.
\end{abstract}

\vspace{0.5cm}

A zonotope is a Minkowski sum of line segments (generators) in Euclidean space. In two dimensions, they are exactly the centrally symmetric polytopes; in higher dimensions, they are still centrally symmetric but aren't uniquely characterized by this property \cite{McM70}. We denote $\ZZ_n(\R^d)$ to be the space of all zonotopes that can be expressed with $n$ generators in $\R^d$. We are interested how well zonotopes can approximate other polytopes under the Hausdorff distance. Given any two polytopes $P$ and $Q$, the Hausdorff distance between $P$ and $Q$ is:
\[ d(P,Q) = \max\left( \sup_{p \in P} d(p, Q), \sup_{q \in Q} d(q,P) \right) \]
Given a fixed polytope $P \subset \R^d$, we define $d_P: \ZZ_n(\R^d) \to \R$ by $Z \mapsto d(P,Z)$. We seek a global minimum of this function, which we call the \emph{Hausdorff optimal approximation to $P$}. We will see that $\ZZ_n(\R^d)$ is locally isomorphic to Euclidean space, and that $d_P$ is a non-convex, piecewise smooth function on that space. 

To approach optimizing $d_P$, in Section \ref{sec:local} we will study the function locally in regimes where we can write it explicitly as a maximum of a finite number of smooth functions. This local theory will allow us to compute gradients and subgradients of this function. In Section \ref{sec:opt}, we will use this local theory to construct a ``feasibility cone'' $C(Z)$ for each zonotope $Z$ in a generic family of zonotopes in $\ZZ_n(\R^d)$. We will show that, under certain circumstances, if the feasiblity cone $C(Z)$ has nonempty interior, then $Z$ is not a local minimum of $d_P$.

Finally, in Section \ref{sec:alg}, we will combine these local results to build an optimization algorithm to find critical points of $d_P$. This algorithm is based on the subgradient method for non-convex non-smooth functions \cite{Shor85}.
Numerical and complexity results for this algorithm in various dimensions are given in Section \ref{sec:numerical}. A link to our python implementation of this algorithm is also provided.

\section{Motivation and Previous Work}
There are several motivations for approximating a polytope by a zonotope. The first is that zonotopal approximation has been shown to be a sub-problem for pruning of ReLU neural networks \cite{MSRM22}. The authors showed that the support of the convex dual of a convex neural network is a zonotope, and to prune such a neural network one can approximate that zonotope with a lower rank zonotope. Zonotopes are also used in 3-D graphics for object collision detection \cite{GNZ03}.  Previous work by \cite{GNZ03} established algorithms to find tight enclosing zonotopes of a polytope. Using these algorithms, collision detection between arbitrary polytopes can be made more efficient by performing collision detection between their enclosing zonotopes (at the cost of some false positives). Our work here can be applied: performing collision detection on the Hausdorff optimal zonotope can give more accurate results, though it will also introduce false negatives as well.

Approximating polytopes and point clouds by zonotopes also has applications to state estimation \cite{BAC06,COM15}, reachability analysis \cite{Ku98,ASB08,Gir05}, and many other areas. These applications arise because of several computationally advantageous properties of zonotopes, like their ability to represent $\O(n^2)$ points with $\O(n)$ information, having efficient set membership algorithms, and being closed under Minkowski sum and multiplication. In some cases, $P$ itself is also a zonotope and one wishes to approximate it by a zonotope of lower rank -- this is called zonotope order reduction (or rank reduction) \cite{YS18}. \\

A central difficulty in optimizing $d_P$ is that it is generally not convex. For example, suppose that that $P\subset \R^2$ has a discrete rotational symmetry $R \in \text{O}(2)$ of order $3$. Then if $Z\in \ZZ_2(\R^2)$ is a local minimum, then so is $R^jZ$ for any $j$. Moreover, $Z$ is a parallelogram and hence cannot have a rotational symmetry of order different from $2$ or $4$. Therefore $RZ \neq Z$, and so $RZ$ and $Z$ are two distinct local minima for $d_P$. It is also a non-smooth function, as we will see. The optimization of non-convex non-smooth functions is an expansive and nontrivial body of problems for which there is no single best method \cite{RW98}.

Our primary contribution is an iterative algorithm to approximate any polytope $P$ by a zonotope $Z \in \ZZ_n(\R^d)$ by optimizing the Hausdorff distance $d_P(Z)$ through subgradient descent. This procedure is detailed in Algorithm \ref{alg:subgradient}. The space complexity of Algorithm \ref{alg:subgradient} is $\O(d { n \choose d-1 })$ and the time complexity per iteration step is equal to the time complexity of computing $d_P(Z)$ for a general position zonotope $Z \in \ZZ_n(\R^d)$. In other words, we show that, given some precomputation, the iterative step for optimizing $d_P$ is bounded by simple evaluation of $d_P$ at a point (at the cost of exponential space complexity). In addition, this work provides a theoretical framework for understanding $d_P$, especially its local properties.

This work is organized as follows. In Sections \ref{sec:zono} and \ref{sec:haus}, we set up the problem by defining the space of zonotopes and establishing some geometric properties of $d_P$. Then in Section \ref{sec:local} we study $d_P$ locally in regimes where we can write it explicitly as a maximum of a finite number of differentiable functions. This local theory will allow us to compute gradients and subgradients of this function. In Section \ref{sec:opt}, we will use this local theory to construct a ``feasibility cone'' for each zonotope in a generic family of zonotopes in $\ZZ_n(\R^d)$. We will show that any interior point of the feasibility cone $\mathcal{C}(P,Z)$ corresponds to perturbations of $Z$ that decrease the Hausdorff distance $d_P(Z)$.

Finally, in Section \ref{sec:alg}, we will combine these local results into Algorithm \ref{alg:subgradient}, which is our proposed method for finding local minima of $d_P$. This algorithm is based on the subgradient method for locally Lipschitz functions \cite{Shor85}. Numerical results and complexity analysis for this algorithm in various dimensions are given in Section \ref{sec:numerical}.

\section{The Space of Zonotopes}
\label{sec:zono}

A zonotope is a Minkowski sum of line segments, also called generators, in Euclidean space:
\[ Z = \sum_i S_i \]
For a standard reference on zonotopes, see for example \cite{GO97,McM70}. We take the convention that every segment $S_i$ has the origin as an endpoint (whereas some authors assume each $S_i$ is centered on the origin, for example \cite{McM71}).

\begin{lem}\label{lem:zonotope}
 Let $I^n\subset \R^n$ be the standard unit cube. If $A: \R^n \to \R^d$ is any affine map, then $A(I^n)$ is a zonotope in $\R^d$. Likewise, any zonotope $Z$ is the image of some unit cube under an affine map $A$. Moreover, if $C^n$ is the set of vertices of $I^n$, then $A(C^n)$ contains the set of vertices of $A(I^n)$.
\end{lem}

\begin{defn}\label{def:cubical}
With $Z,A$ and $C^n$ as in Lemma \ref{lem:zonotope}, we call $A(C^n)$ the \emph{cubical vertices} of $Z$. We call $A(\mathbf{0})$ the \emph{base point} or \emph{translation vector} of $Z$. We say also say that $Z$ has \emph{rank} $n$ if $n$ is the smallest integer such that $Z = A(I^n)$.\footnote{Some authors refer to the rank as the ``number of zones''.}
\end{defn}

It will be useful for us to be able to characterize when a cubical vertex of a zonotope $Z$ is a vertex of $Z$. The following lemma formulates this as a condition on the generators of $Z$.

\begin{lem}\label{lem:lp}
Let $Z\subset \R^d$ be a zonotope, let $V$ be its set of cubical vertices, and let $\{g_i\}_{i=1,...,n}$ be its generators. For any $m \in V$, if we write $m = \sum_{i\in J}g_i$, then $m$ is a vertex of $Z$ if and only if there exists a hyperplane containing the origin $H \subset \R^d$ that separates $\{g_i\}_{i\in J}$ and $\{g_i\}_{i\notin J}$.
\end{lem}
\begin{proof}
Suppose such a hyperplane $H$ exists and let $\hat{n}$ be the unit normal vector to $H$ such that $\langle g_i,\hat{n}\rangle > 0$ for all $i \in J$. Then define the linear map $L: \R^d \to \R$ by $L(x) = \langle x,\hat{n}\rangle$. Note that $\argmax_{m' \in V} L(m') = m$ because adding or removing a generator from $m$ will strictly decrease the value of $L$. Since a linear function cannot be maximized at a non-vertex point of a polytope, and since $\VERT(Z) \subset V$, see that:
\[ \max_{x \in Z} L(x) = \max_{v \in \VERT(Z)} L(x) \leq \max_{m' \in V} L(m')\]
We have already concluded that the RHS is $L(m)$, so in fact $L$ is maximized over all of $Z$ at $m$. Therefore $m$ is a vertex.

Conversely, if $m$ is a vertex, then there is a linear function $L$ that is maximized strictly at $m$. Let $J' = \{i \in [n] \mid L(g_i)>0\}$. By the same argument as above, $L$ is maximized at $m' = \sum_{i \in J'} g_i$. Since we assumed this maximum was uniquely $m$, we have $m=  m'$ and hence $J'= J$. Therefore $H = \ker(L)$ separates $\{g_i\}_{i \in J}$ from $\{g_i\}_{i\notin J}$.
\end{proof}

For the remainder of this work, we establish the following notation. Given any zonotope $Z$, we denote the affine map associated to $Z$ to be $A_Z$. We denote the linear component of the affine map $A_Z$ in the standard basis by a matrix $Q_Z \in \R^{n \times d}$ and we denote the translation component of $A_Z$ by a vector $\mu_Z \in \R^d$. Conversely, given a matrix $Q \in \R^{n \times d}$ and a vector $\mu \in \R^{d}$, we denote the zonotope associated to the affine map $A(x) = Qx + \mu$ by $Z(Q,\mu)$.

An alternative characterization of a zonotope that we will make use of is the symmetry of its faces. 
\begin{defn}
A set of points is called \emph{centrally symmetric} if it is invariant under point reflection around its barycenter. 
\end{defn}
The faces of a zonotope are centrally symmetric, and in fact, this is also a sufficient condition:
\begin{prop}[\cite{McM70}]\label{prop:2faces}
A polytope $X$ is a zonotope if and only if all 2-faces are centrally symmetric.
\end{prop}

We call a zonotope whose generators are in general position a \emph{general position zonotope}. These enjoy a non-degeneracy property:
\begin{lem}\label{lem:zono:determined}
If $Z$ is a general position zonotope, then the set of 1-faces (edges) of $Z$ uniquely determine the generators of $Z$ (and hence also the rank of $Z$).
\end{lem}

\begin{proof}
Since all vertices of $Z$ are sums of generators, and two vertices of $Z$ have an edge between them if and only if their respective sums differ by one generator, it follows that the set of all edges of $Z$ are translated generators, perhaps with a minus sign. Thus we only need to determine the appropriate sign for each putative generator coming from an edge. Let $\{ \sigma_1 g_1, ... , \sigma_n g_n\}$ be the set of distinct putative generators coming from the edges of $Z$, where $\sigma_i \in \{-1,1\}$ are unknown. There are $2^n$ possible zonotopes obtained by varying $\sigma_i$. Since $Z$ is general position, these are all distinct zonotopes. So we may simply exhaust all possibilities for $\sigma_i$ to find a zonotope which is a translation of our original $Z$. 
\end{proof}

To properly define the space of zonotopes $\ZZ_n(\R^d)$, we will parameterize it using the generator matrix $Q$ and the translation vector $\mu$. With this in mind, we define:
\[ \widetilde{\ZZ_n}(\R^d) = \{ (Q,\mu) \in \R^{n \times d} \times \R^d\} \]
Since the symmetric group $S_n$ acts on this space by permuting the columns of $Q$, and the ordering of these columns does not affect the underlying zonotope $Z(Q,\mu)$, we also quotient by this action. Thus we define the space of zonotopes as:
\[ \ZZ_n(\R^d) := \widetilde{\ZZ_n}(\R^d) / S_n \]
In practice, given a zonotope $Z \in \ZZ_n(\R^d)$, we choose the representative $(Q,\mu)$ such that the columns of $Q$ are lexicographically ordered.


\begin{prop}
Let $Z$ be a zonotope such that $Z = Z(Q,\mu)$ and $Z = Z(Q',\mu')$, where the columns of $Q$ and $Q'$ are in general position, respectively. Then $\mu = \mu'$ and the columns of $Q$ are obtained by permuting the columns of $Q'$.
\end{prop}
\begin{proof}
Let $I^n \subset \R^n$ be the unit cube. By definition, we have $Q(I^n) + \mu = Q'(I^n) + \mu'$. Since $0 \in I^n$, this implies $\mu = \mu'$. Since $Z$ is a general position zonotope, its set of generators is uniquely determined, by Lemma \ref{lem:zono:determined}. Therefore the columns of $Q$ and $Q'$ are the same up to permutation.
\end{proof}

\section{The Hausdorff Distance}
\label{sec:haus}

Recall we are fixing a polytope $P \subset \R^d$ and wish to optimize the function:
\[ d_P : \ZZ_n(\R^d) \to \R \]
\[ Z \mapsto d(P,Z) \]
The Hausdorff distance between polytopes has some computationally convenient properties that we describe here. In general, for any convex set $Q$, the function $x \mapsto d(x,Q)$ is convex \cite{Kon14}. Since a convex function on a polytope is achieved at a vertex, the value of $\sup_{p\in P} d(p,Q)$ is achieved at a vertex of $P$. For future convenience, for any closed set $W \subset \R^d$ and point $p \in \R^d$, we denote $\Pi(p,W) := \argmin_{w \in W} \lN p-w \rN_2$, which is the projection of $p$ onto $W$. In this setting, the Hausdorff distance between two polytopes $P$ and $Q$ is:
\begin{align}\label{eq:Hausdorff-distance}
d(P,Q) &=  \max\left( \max_{p \in \VERT(P)} d(p,Q), \max_{q \in \VERT(Q)} d(q,P) \right) 
\end{align}
Computationally, this means that calculating the Hausdorff distance is reduced to computing the distance between a point and a polytope. This can be formulated as a quadratic program, and hence be solved in polynomial time. For an account of this, consult e.g. \cite{Kon14}. The following similar expression for the Hausdorff distance between two polytopes will be useful for us:
\begin{prop}
Let $P$ and $Q$ be polytopes. For every vertex $p \in \VERT(P)$, denote $W_p$ to be the affine hull of the smallest face\footnote{Keep in mind $Q$ itself is also a face of $Q$. So if $p$ is in the interior of $Q$, then the smallest face of $Q$ containing $p$ is $Q$ itself.} of $Q$ containing $\Pi(p,Q)$. For $q \in \VERT(Q)$, we define $W_q$ equivalently. Then:
\begin{equation}\label{eq:alt-Hausdorff}
d(P,Q) = \max\left( \max_{p \in \VERT(P)} d(p, W_p), \max_{q \in \VERT(Q)} d(q, W_q) \right)
\end{equation}
where $d(x,W)$ is the distance between a point $x$ and an affine subspace $W$
\end{prop}
\begin{proof}
This is immediate from (\ref{eq:Hausdorff-distance}) and the fact that $d(p,W_p)$ is the same as the distance between $p$ and the smallest face of $P$ containing $p$.
\end{proof}

We will also consider a slightly simpler function $d^c_P: \ZZ_n(\R^d) \to \R$ which we call the \emph{coarse Hausdorff distance}. For any polytopes $P$ and $Q$, we define:
\begin{align}
d^c(P,Q) &:= d(\VERT(P),\VERT(Q))  \\
&= \max\left( \max_{p\in \VERT{P}} \min_{q \in \VERT{Q}} d(p,q), \max_{q'\in \VERT{Q}} \min_{p' \in \VERT{P}} d(p',q') \right) \label{eq:dc}
\end{align}
whence $d_P^c(Z) := d_P^c(P,Z)$. In this case, we are only considering the vertices of $P$ and $Z$ to measure the distance between them. It is clear from the definition that $d_P \leq d_P^c$. Though $d_P^c$ is an upper bound for $d_P$, we will see in Section \ref{sec:opt} that $d_P^c$ can have local minima where $d_P$ does not. Thus we cannot simply minimize $d_P^c$ to minimize $d_P$.

We establish the following convention for $d(P,Q)$. Hereafter, whenever we say a pair $(p,q) \in P \times Q$ achieves the Hausdorff distance between $P$ and $Q$, we mean that $d(P,Q) = \lN p-q\rN_2$ and at least one of $p$ or $q$ is a vertex of its respective polytope. Since $P$ and $Q$ are polytopes, the Hausdorff distance is always achieved at at least one pair in this manner.

Given a polytope $P \subset \R^d$ and a point $x \in \R^d$, any point that achieves the distance $d(x,P)$ has the following useful characterization:

\begin{lem}
Let $P\subset \R^d$ be a polytope and $x \in \R^d$. Then for any $p\in P$, we have $\lN x-p\rN_2 = d(x,P)$ if and only if $x-p$ is contained in the normal cone of $p$.
\end{lem}
\begin{proof}
See \cite{Bur14}, Theorem 5.21.
\end{proof}

\begin{defn}
For a polytope $P \subset \R^d$ and a point $x \notin P^\circ$, we say that $x$ is \emph{Hausdorff stable relative to $P$} if $x$ is contained in the relative interior of the normal cone to $P$ at $\Pi(x,P)$ (see Figure \ref{fig:Hausdorff-stable}). If $x \in P^\circ$, we also say that it is Hausdorff stable, for reasons explained in the next remark.
\end{defn}

\begin{rem}\label{rem:Hausdorff-stable}
A point that is Hausdorff stable relative to $P$ has the property that there exists an open neighborhood $U \ni x$ such that for any $y \in U$, the smallest face of $P$ containing $\Pi(y,P)$ is the same as the smallest face of $P$ containing $\Pi(x,P)$. It is easy to see that Hausdorff stability is an open dense condition on $\R^d$, and it will be a necessary property in the next section when we develop a local theory of the Hausdorff distance $d_P$.
\end{rem}

Hausdorff stability is closely related to the normal fan of $P$ in the following way. If we define the \emph{normal complex} of $P$ to be the complex obtained by attaching to every $x \in \del P$ the normal cone at $x$, then the points in the skeleton of this complex are exactly the ones that are not Hausdorff stable.

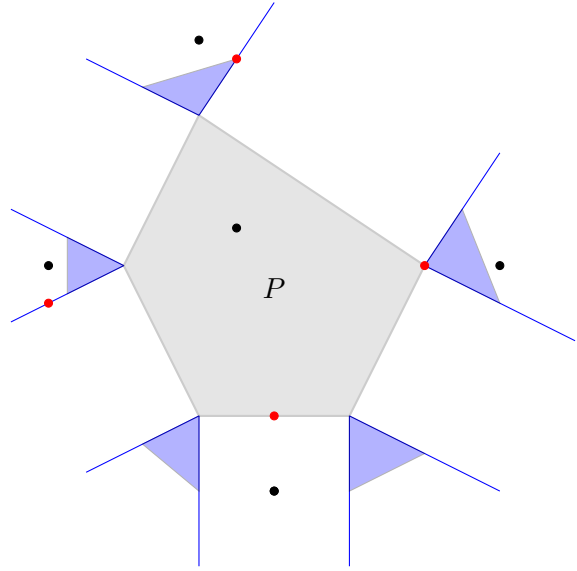
\begin{figure}
\begin{center}
\begin{tikzpicture}
\draw[fill=gray,opacity=0.2,thick] (0,0) -- (2,0) -- (3,2) -- (0,4) -- (-1,2) -- (0,0) -- cycle;
\draw[blue] (0,0) -- (0,-2);
\draw[blue] (0,0) -- (-1.5,-0.75);
\draw[fill=blue,opacity=0.3] (0,0) -- (0,-1) -- (-0.75,-0.375) -- cycle;
\draw[blue] (-1,2) -- (-1-1.5,2-0.75);
\draw[blue] (-1,2) -- (-2.5,2.75);
\draw[fill=blue,opacity=0.3] (-1,2) -- (-1.75,2.375) -- (-1-0.75,2-0.375) -- cycle;
\draw[blue] (0,4) -- (-2.5+1,2.75+2);
\draw[blue] (0,4) -- (1,5.5);
\draw[fill=blue,opacity=0.3] (0,4) -- (0.5,4.75) -- (-0.75,4.375) -- cycle;
\draw[blue] (3,2) -- (1+3,5.5-2);
\draw[blue] (3,2) -- (5,1);
\draw[fill=blue,opacity=0.3] (3,2) -- (3.5,2.75) -- (4,1.5) -- cycle;
\draw[blue] (2,0) -- (5-1,1-2);
\draw[blue] (2,0) -- (2,-2);
\draw[fill=blue,opacity=0.3] (2,0) -- (3,-0.5) -- (2,-1) -- cycle;
\node at (1,1.7) {$P$};
\draw[fill=red,draw=red] (1,0) circle (1.5pt);
\draw[fill=red,draw=red] (0.5,4.75) circle (1.5pt);
\draw[fill=red,draw=red] (-2,1.5) circle (1.5pt);
\draw[fill=red,draw=red] (3,2) circle (1.5pt);
\draw[fill=black,draw=black] (0.5,2.5) circle (1.5pt);
\draw[fill=black,draw=black] (-2,2) circle (1.5pt);
\draw[fill=black,draw=black] (1,-1) circle (1.5pt);
\draw[fill=black,draw=black] (1,-1) circle (1.5pt);
\draw[fill=black,draw=black] (0,5) circle (1.5pt);
\draw[fill=black,draw=black] (4,2) circle (1.5pt);
\end{tikzpicture}
\end{center}
\caption{A polytope with normal cones at each vertex attached (blue). All points colored red are not Hausdorff stable relative to $P$ and all points colored black are Hausdorff stable relative to $P$.}
\label{fig:Hausdorff-stable}
\end{figure}

\section{Local Theory of $d_P$}
\label{sec:local}

In this section, we will show how to express $d_P$ as an explicit maximum of differentiable functions in a neighborhood of a generic point $Z \in \ZZ_n(\R^d)$. To do this, we will need to show that, given a sufficiently small perturbation $Z \to Z'$, we can identify the boundaries of $Z$ and $Z'$ naturally and therefore express the Hausdorff distance as a function of the generators of $Z$. This will allow us to calculate the gradient (or subgradient) of $d_P$ and, in the next section, will also give us the tools necessary to come up with direction finding criteria for gradient descent on $d_P$.

Our intent is to identify points on the boundary of a zonotope $Z \in \ZZ_n(\R^d)$ with their corresponding points in the unit cube $I^n$ that map to them (i.e. the lift of $\del Z$ to $I^n$). If we perturb $Z \to Z'$ by a small amount, we then might be able to identify the boundary of $Z$ with the boundary of $Z'$ via their lifts. More precisely, given $q \in \del Z$, let $\Lift(q):=A_Z^{-1}(q) \cap I^n$ be the set of points in the unit cube that map to $q$. If $\Lift(q)$ happens to be a unique point, then given any other zonotope $Z'  \in \ZZ_n(\R^d)$, there is a well-defined mapping $\varphi: \del Z \to Z'$ defined by $\varphi(q) =A_{Z'}(\Lift(q))$. We call $\varphi: \del Z \to Z'$ the \emph{pushforward map}. It identifies $\del Z$ with a subset of $Z'$ in a natural way.

\begin{defn}
Given a zonotope $Z \in \ZZ_n(\R^d)$, we say that a point $q \in \del Z$ is \emph{stable} if $\Lift(q)$ is a single point. We say that $Z$ is stable if every point on its boundary is stable.
\end{defn}

\begin{prop}
If $Z \in \ZZ_n(\R^d)$ is a general position zonotope, then $Z$ is stable.
\end{prop}
\begin{proof}
We first claim that all vertices of $Z$ are stable. It suffices to assume that $A_Z$ is linear since translation does not affect stability. Let $v \in \VERT(Z)$ and suppose that $\Lift(v)$ contains more than one point. Since $\VERT(Z) \subset A_Z(C^n)$, where $C^n$ are the vertices of $I^n$, there exist $e,e' \in \Lift(v)$ distinct that are in $C^n$. Let $F \subset I^n$ be the smallest face containing $e$ and $e'$. If $F$ is a 1-face, i.e. an edge of $I^n$, then $e-e'$ is a standard unit vector $e_i$ in $\R^n$, up to sign. Since $0=A_Z(e)-A_Z(e') = A_Z(e-e')$, this would imply that the $i$th column of $Q_Z$ is zero. This contradicts our assumption that the generators of $Z$ are in general position. 

Therefore $F$ must be a $k$-face for $k>1$. Let $\ell = [e,e']$ be the line segment with endpoints $e$ and $e'$. The face $F$ is a cube of dimension $k$ and since it is the minimal one containing the vertices $e$ and $e'$, the segment $\ell$ is a diagonal of $F$. Consider the image $A_Z(F)$, which contains $v$. Since the diagonal of $F$ is mapped to a vertex of $A_Z(F)$, it follows that $v$ is a vertex of $A_Z(F)$. This means that there exists another vertex $e'' \in F$ that is mapped to $v$ as well. Letting $F' \subset F$ be the smallest face containing $e''$ and $e'$, we can repeat this process again, and since $F'$ will always be of strictly smaller dimension than $F$, it will terminate at $F'$ being a line segment, which is the $k=1$ case we have already ruled out. Therefore every $v \in \VERT(Z)$ has a unique lift.

Now to show that $Z$ is stable, consider any facet $X \subset Z$ and let $S$ be the set of vertices of $X$. Since $Z$ is a general position zonotope, every facet has exactly $d+1$ vertices. Therefore $\dim(X) = |S|-1$. Additionally, is straightforward to show that, because of linearity of $A_Z$, we have:
\[ \Lift(\Conv(S)) = \Conv(\Lift(S)) \]
Since every element of $S$ is stable, $|\Lift(S)| = |S|$. Therefore the dimension of $\Lift(F)$ is at most $|S|-1 = \dim(X)$. Therefore every element of $X$ has a unique lift, and so $Z$ is stable.
\end{proof}

\begin{cor}
If $q \in \del Z$ is stable, then so is every other point on the minimal face of $Z$ containing $q$.
\end{cor}

\begin{cor}\label{cor:stable}
If $Z$ is stable and $q \in \del Z$ is a cubical vertex, then $q$ is a vertex of $Z$.
\end{cor}

The pushforward map $\varphi: \del Z \to Z'$ is only defined when $Z$ is stable. When the image of the pushforward $\varphi$ lies in $\del Z'$, we say it is \emph{proper}. For general pairs $Z,Z'$ the pushforward will not be proper, as in the next example.

\begin{exmp}
Fix $\epsilon \geq 0$ and let $Q = \begin{pmatrix} 1 & 1 & 2 \\ 2 & 1 & 0 \end{pmatrix}$ and $\Delta Q = \begin{pmatrix} 0 & -\epsilon & 0 \\ 0 & 0 & 0 \end{pmatrix}$. These define zonotopes $Z = Z(Q,0)$ and $Z' = Z(Q+ \Delta Q, 0)$. The image of the pushforward $\varphi: \del Z \to Z'$ for various values of $\epsilon$ is shown in Figure \ref{fig:pushforward}. We see that for $\epsilon \geq \frac{1}{2}$, the pushforward is not proper.
\end{exmp}

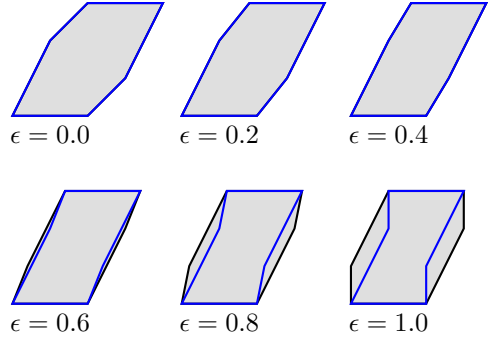
\begin{figure}
\begin{center}
\begin{tikzpicture}[scale=0.5]
\draw[fill=gray,thick,fill opacity=0.25] (4.0, 3.0) -- (2.0, 3.0) -- (1.0, 2.0) -- (0.0, 0.0) -- (2.0, 0.0) -- (3.0, 1.0) -- cycle;
\draw[thick,blue] (0.0, 0.0) -- (1.0, 2.0) -- (2.0, 3.0) -- (4.0, 3.0) -- (3.0, 1.0) -- (2.0, 0.0) -- (0.0, 0.0) -- cycle;
\node at (1.0, -0.5) {\small{$\epsilon=0.0$}};
\draw[fill=gray,thick,fill opacity=0.25] (8.3, 3.0) -- (6.3, 3.0) -- (5.5, 2.0) -- (4.5, 0.0) -- (6.5, 0.0) -- (7.3, 1.0) -- cycle;
\draw[thick,blue] (4.5, 0.0) -- (5.5, 2.0) -- (6.3, 3.0) -- (8.3, 3.0) -- (7.3, 1.0) -- (6.5, 0.0) -- (4.5, 0.0) -- cycle;
\node at (5.5, -0.5) {\small{$\epsilon=0.2$}};
\draw[fill=gray,thick,fill opacity=0.25] (12.6, 3.0) -- (10.6, 3.0) -- (10.0, 2.0) -- (9.0, 0.0) -- (11.0, 0.0) -- (11.6, 1.0) -- cycle;
\draw[thick,blue] (9.0, 0.0) -- (10.0, 2.0) -- (10.6, 3.0) -- (12.6, 3.0) -- (11.6, 1.0) -- (11.0, 0.0) -- (9.0, 0.0) -- cycle;
\node at (10.0, -0.5) {\small{$\epsilon=0.4$}};
\draw[fill=gray,thick,fill opacity=0.25] (3.4, -2.0) -- (1.4, -2.0) -- (0.3999999999999999, -4.0) -- (0.0, -5.0) -- (2.0, -5.0) -- (3.0, -3.0) -- cycle;
\draw[thick,blue] (0.0, -5.0) -- (1.0, -3.0) -- (1.4, -2.0) -- (3.4, -2.0) -- (2.4, -4.0) -- (2.0, -5.0) -- (0.0, -5.0) -- cycle;
\node at (1.0, -5.5) {\small{$\epsilon=0.6$}};
\draw[fill=gray,thick,fill opacity=0.25] (7.7, -2.0) -- (5.7, -2.0) -- (4.7, -4.0) -- (4.5, -5.0) -- (6.5, -5.0) -- (7.5, -3.0) -- cycle;
\draw[thick,blue] (4.5, -5.0) -- (5.5, -3.0) -- (5.7, -2.0) -- (7.7, -2.0) -- (6.7, -4.0) -- (6.5, -5.0) -- (4.5, -5.0) -- cycle;
\node at (5.5, -5.5) {\small{$\epsilon=0.8$}};
\draw[fill=gray,thick,fill opacity=0.25] (9.0, -5.0) -- (11.0, -5.0) -- (12.0, -3.0) -- (12.0, -2.0) -- (10.0, -2.0) -- (9.0, -4.0) -- cycle;
\draw[thick,blue] (9.0, -5.0) -- (10.0, -3.0) -- (10.0, -2.0) -- (12.0, -2.0) -- (11.0, -4.0) -- (11.0, -5.0) -- (9.0, -5.0) -- cycle;
\node at (10.0, -5.5) {\small{$\epsilon=1.0$}};

\end{tikzpicture}
\caption{The zonotope $Z'$ for various values of $\epsilon$, with the image of the pushforward $\varphi: \del Z \to Z'$ shown in blue.}
\label{fig:pushforward}
\end{center}
\end{figure}

\begin{lem}\label{lem:stable}
Suppose that $Z = Z(Q,\mu) \in \ZZ_n(\R^d)$ is stable. Then there exists $\epsilon > 0$ such that for all perturbations $Z' = Z(Q+ \Delta Q, \mu + \Delta \mu)$ satisfying $\lN \Delta Q \rN_2 < \epsilon$ the pushforward map $\varphi: \del Z \to Z'$ is proper.
\end{lem}
\begin{proof}
We first assume that $\Delta \mu = 0$. Our strategy will be to first show that $\varphi(\VERT(Z)) = \VERT(Z')$, from which the desired result will follow.

Since $Z$ is stable, all cubical vertices which are not vertices are in the interior of $Z$, by Corollary \ref{cor:stable}. For each cubical vertex $v \in Z$, let $e_v \in C^n$ be the corresponding lift, and denote $E_{\Int} = \{ e \in C^n \mid e \neq e_v \ \  \forall v \in \VERT(Z)\}$. For any $\epsilon > 0$, let the $\epsilon$-fattening of $\del Z$ be:
\[ (\del Z)_{\epsilon} := \bigcup_{\lN \Delta Q \rN_2 < \epsilon} \del Z'\]
where $Z' = Z(Q+ \Delta Q, \mu)$. Now define:
\[ \tau(\epsilon) = \min_{e \in E_{\Int}} d(Qe+\mu, (\del Z)_\epsilon)\]
This represents the minimum distance of all interior cubical vertices of $Z$ to the boundaries of all possible $\epsilon$ perturbations of $Z$. Since all cubical vertices are on the interior of $Z$, we can pick $\epsilon$ small enough so that the fattening $(\del Z)_\epsilon$ does not contain any of them. In other words, there exists $\epsilon_0>0$ such that $\tau(\epsilon_0) > 0$.

Now consider the function $M: \R^{d\times n} \to \R$ given by:
\[ \Delta Q \mapsto \min_{e \in E_{\Int}} \lN \Delta Q e \rN_2 \]
This is continuous, so we can find $\epsilon_1 \leq \epsilon_0 $ such that $\lN \Delta Q \rN_2 < \epsilon_1 \Rightarrow M(\Delta Q) < \tau(\epsilon_0)$. Given such a $\Delta Q$, the images of the interior cubical vertices $e \in E_{\Int}$ are in the interior of the perturbed zonotope $Z' = Z(Q + \Delta Q, \mu)$. This means that any vertex of $Z'$ must be the pushforward of some vertex in $Z'$, i.e. $\VERT(Z') \subseteq \varphi(\VERT(Z))$.

Now we will show the reverse, by considering the vertices of $Z$. For each $v \in \VERT(Z)$, write $v = \sum_{j \in J_v} g_j$, where $g_j$ is the $j$th generator of Z (row of $Q$). Because the generators of $Z$ are in general position, there exists a hyperplane $H_v$ separating $\{g_j\}_{j\in J_v}$ from $\{g_j\}_{j\notin J_v}$. Let $\delta =\min_{v\in \VERT(Z)} \min_{j} d(g_j, H_v)$. Then there exists $\epsilon_2$ such that $\lN \Delta Q\rN_2 < \epsilon_2 \Rightarrow \min_j \lN \Delta g_j \lN_2 < \delta $, where $\Delta g_j$ is the $j$th row of $\Delta Q$. Thus if $\lN \Delta Q\rN_2 < \epsilon_2$, the hyperplanes $H_v$ still separate the appropriate generators of $Z' = Z(Q + \Delta Q, \mu)$ and therefore $\varphi(\VERT(Z)) \subseteq \VERT(Z')$.

We have thus shown that $\varphi(\VERT(Z)) = \VERT(Z')$ when $\lN \Delta Q \rN_2 < \min(\epsilon_1,\epsilon_2)$. Now consider any facet $F \subset \del Z$. Since the vertices of $Z$ contained in $F$ all map to the boundary of $Z'$, we must also have $\varphi(F) \subset \del Z'$. Applying this reasoning to all facets of $\del Z$ we find that $\varphi(\del Z) \subset \del Z'$.

Finally, to handle the case when $\Delta \mu \neq 0$, take $\Delta Q = 0$. Clearly $\varphi: \del Z \to Z'$ lands in the boundary because $\varphi$ is given by translation by $\Delta \mu$. Now for any perturbation $Z' = Z(Q + \Delta Q, \mu + \Delta \mu)$, we can write it as a composition of perturbations $Z \to Z' \to Z''$, the first of which with $\Delta \mu = 0$ and the second of which with $\Delta Q = 0$. If we assume $\lN \Delta Q \rN_2 < \min(\epsilon_1,\epsilon_2)$, the resulting pushforward $\varphi: \del Z \to Z''$ is well-defined and its image lies in the boundary of $Z''$.

\end{proof}

Having a proper pushforward $\varphi: \del Z \to Z'$ in a neighborhood of a zonotope $Z$ is useful because the boundary of all zonotopes in that neighborhood are combinatorially the same: they are all the image of the same subset of the unit cube in $\R^n$. This will allow us to write down a well-defined, explicit function that is equal to $d_P$ in some neighborhood of $Z$.

The following theorem will serve as a foundation for our method of computing subgradients of $d_P$ as well as allow us to define our optimization algorithm for $d_P$. We will frequently require the following assumptions on a zonotope $Z \in \ZZ_n(\R^d)$ and a polytope $P \subset \R^d$:

\begin{enumerate}
\item $Z$ is a general position zonotope (and therefore stable).
\item Each vertex of $P$ is Hausdorff stable relative to $Z$ and each vertex of $Z$ is Hausdorff stable relative to $P$.
\end{enumerate}
We will refer to these as the \emph{locality conditions} on $P$ and $Z$. For fixed $P$, these are both open and dense conditions on $Z$.


\begin{thm}\label{thm:local}
Fix $P \subset \R^d$ a polytope and suppose $Z_0 \in \ZZ_n(\R^d)$ satisfies the locality conditions 1) and 2) above. Then there exists an open neighborhood $U$ of $Z_0$ such that for all $Z \in U$:
\begin{equation}\label{eq:localdp} d_P(Z) = \max \left( \max_{p \in \VERT{P}} d(p, \Aff(\varphi_Z(F_p)) ), \max_{q \in \VERT{Z_0}} d(\varphi_Z(q), \Aff(F_q)) \right) \end{equation}
where:
\begin{itemize}
\item $\varphi_Z: \del Z_0 \to \del Z$ is the pushforward.
\item $d(x, W)$ is the distance between $x$ and an affine subset $W$.
\item $F_p$ is the smallest face of $Z$ containing $\Pi(p,Z)$, and $F_q$ is the smallest face of $P$ containing $\Pi(q,P)$.
\item $\Aff(X)$ is the affine hull of a set $X$.
\end{itemize}
Moreover, each term in the maximum above is a square root of bounded rational functions in the parameters $(Q,\mu)$ defining $Z$. In particular, they are differentiable functions of $Z$ inside $U$.
\end{thm}

\begin{proof}

First, we observe that the RHS of (\ref{eq:localdp}) evaluated at $Z_0$ is the RHS of (\ref{eq:alt-Hausdorff}) because $\varphi_{Z_0}$ is the identity. Therefore (\ref{eq:localdp}) holds at $Z_0$. Now we let $V \ni Z_0$ be an open neighborhood such that $\varphi_Z$ is proper for all $Z \in V$ (Lemma \ref{lem:stable}). Recall Remark \ref{rem:Hausdorff-stable}, which says that Hausdorff stability of all vertices implies that the vertices of $Z_0$ can be perturbed slightly and the faces $F_q$ for $q \in \VERT(Z_0)$ remain unchanged. For similar reasons, given $p \in \VERT(P)$ the smallest face containing $\Pi(p, Z_0)$ after a small perturbation of $Z_0$ is still $F_p$. In other words, there exists $U \subset V$ open containing $Z_0$ such that for all $Z \in U$:
\begin{itemize}
\item The smallest face of $P$ containing $\Pi(\varphi_Z(q),P)$ is $F_q$, and 
\item The smallest face of $Z$ containing $\Pi(p, Z)$ is $\varphi_Z(F_p)$.
\end{itemize}
Now applying equation (\ref{eq:alt-Hausdorff}), we see that the RHS of (\ref{eq:localdp}) holds for $Z$.

Finally, we relegate the proof that the functions $d(p,\Aff(\varphi_Z(F_p)))$ and $d(\varphi_Z(q), \Aff(F_q))$ are square roots of bounded rational functions to Section \ref{subsec:computation}, specifically in Corollaries \ref{cor:grad1} and \ref{cor:grad2}. 
\end{proof}

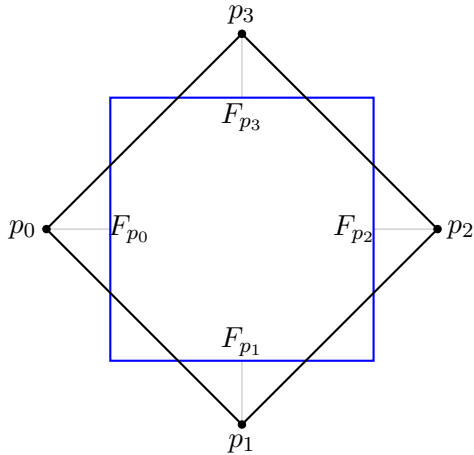
\begin{figure}
\begin{center}
\begin{tikzpicture}[scale=3.5]

\draw[opacity=0.5,gray] (-0.24246212024587482, 0.5004419417382416) -- (-3.736178618729002e-21, 0.5004419417382415);
\node at (0.07273863607376245, 0.5004419417382415) {$F_{ p_{ 0 } }$};
\draw[fill=black] (-0.24246212024587482, 0.5004419417382416) circle (0.4pt);
\node[anchor=east] at (-0.24246212024587482, 0.5004419417382416) {$p_0$};
\draw[opacity=0.5,gray] (0.49999999999999994, -0.24246212024587482) -- (0.4999999999999999, -3.9908782242322705e-21);
\node at (0.4999999999999999, 0.07273863607376245) {$F_{ p_{ 1 } }$};
\draw[fill=black] (0.49999999999999994, -0.24246212024587482) circle (0.4pt);
\node[anchor=north] at (0.49999999999999994, -0.24246212024587482) {$p_1$};
\draw[opacity=0.5,gray] (1.2429040619841165, 0.5004364174665135) -- (1.0, 0.5004364174665135);
\node at (0.9271287814047651, 0.5004364174665135) {$F_{ p_{ 2 } }$};
\draw[fill=black] (1.2429040619841165, 0.5004364174665135) circle (0.4pt);
\node[anchor=west] at (1.2429040619841165, 0.5004364174665135) {$p_2$};
\draw[opacity=0.5,gray] (0.499994475728272, 1.242893082494057) -- (0.499994475728272, 1.0);
\node at (0.499994475728272, 0.9271320752517829) {$F_{ p_{ 3 } }$};
\draw[fill=black] (0.499994475728272, 1.242893082494057) circle (0.4pt);
\node[anchor=south] at (0.499994475728272, 1.242893082494057) {$p_3$};
\draw[thick,color=blue] (0.0, 0.0) -- (1.0, 0.0) -- (1.0, 1.0) -- (0.0, 1.0) -- cycle;
\draw[thick] (-0.24246212024587482, 0.5004419417382416) -- (0.49999999999999994, -0.24246212024587482) -- (1.2429040619841165, 0.5004364174665135) -- (0.499994475728272, 1.242893082494057) -- cycle;

\end{tikzpicture}
\end{center}
\caption{Example of A zonotope (blue) and a polytope (black) where the Hausdorff distance $d_P(Z)$ is achieved at four pairs of points. The corresponding faces $F_{p}$ in $Z$ are labeled. } 
\label{fig:localdp}
\end{figure}

A schematic of this local theory is shown in Figure \ref{fig:localdp}. The functions $d(x,W)$ are differentiable and locally Lipschitz, and since the pointwise maximum of locally Lipschitz functions is also locally Lipschitz, we have:

\begin{cor}
$d_P: \ZZ_n(\R^d) \to \R$ is locally Lipschitz at all points $Z$ that satisfy the locality conditions 1) and 2), and is therefore also differentiable almost everywhere in $\ZZ_n(\R^d)$.
\end{cor}

\section{The Feasibility Cone}
\label{sec:opt}

In this section we will define the feasibility cone $\mathcal{C}(P,Z)$ associated to a polytope $P$ and a zonotope $Z$. We will see that the interior of this cone corresponds to perturbations of $Z$ that decrease the Hausdorff distance $d_P(Z)$. 

\begin{prop}\label{prop:optim-cone}
Fix $P$ a polytope and $Z = Z(Q,\mu) \in \ZZ_n(\R^d)$ that satisfy the locality conditions 1) and 2). Let $d_P(Z)$ be achieved at $k$ distinct pairs of points $(p_1,q_1),...,(p_k,q_k)$. For each $q_i$ let $e_i = \Lift(q_i)$ be the (unique) lift of $q_i$ to $I^n$. Define the set:
\[ \mathcal{C}(P,Z) = \{ (\Delta Q, \Delta \mu) \in \R^{n\times d} \times \R^{d} \mid \langle \Delta Q e_i + \Delta \mu , p_i - q_i \rangle \geq 0 \ \forall i \} \]
Then $\mathcal{C}(P,Z)$ is a polyhedral cone, and for any $(\Delta Q, \Delta \mu)$ in the interior of $\mathcal{C}(P,Z)$, there exists $\epsilon>0$ such that for all $t< \epsilon$ the associated perturbed zonotope $Z_t = Z(Q + t\Delta Q, \mu + t\Delta \mu)$ satisfies $d_P(Z_t) < d_P(Z)$.
\end{prop}
\begin{rem}
We call $\mathcal{C}(P,Z)$ the \emph{feasibility cone} of $Z$ relative to $P$.
\end{rem}
\begin{proof}
First we show that $\CC(P,Z)$ is a polyhedral cone. This can be seen by rewriting the inequalities as:
\[ \langle \Delta Q e_i + \Delta \mu, p_i - q_i \rangle = \langle (\Delta Q, \Delta \mu), (p_i - q_i) \otimes (e_i, 1) \rangle \]
which are linear in $(\Delta Q, \Delta \mu)$.

Let $(\Delta Q, \Delta \mu)$ be an interior point of $\CC(P,Z)$. Then $\langle \Delta Q e_i + \Delta \mu, p_i-q_i\rangle > 0$ for all $i$. For $t > 0$ denote $Z_t = Z(Q + t\Delta Q, \mu + t\Delta \mu)$, and for $T>0$, let $B_T(Z) = \{ Z_t  \ : \ t < T\}$. By Theorem \ref{thm:local}, there is $\epsilon_0 > 0$ be such that $d_P|_{B_{\epsilon_0}(Z)}$ is a maximum of distances between points and affine subspaces (equation \ref{eq:localdp}). We wish to show that there exists $0 < \epsilon \leq \epsilon_0$ so that for all $t < \epsilon$:
\begin{enumerate}
\item all terms achieving the maximum in \ref{eq:localdp} at $Z$ are strictly less than $d_P(Z)$ at $Z_t$, and
\item all terms not achieving the maximum in \ref{eq:localdp} at $Z$ remain less than $d_P(Z)$ at $Z_t$.
\end{enumerate}
It then directly follows that $d_P(Z_t) < d_P(Z)$ for all $t< \epsilon_1$. We will show 1) first. For every $i$, let $\alpha_i$ be such that
\[ \cos(\alpha_i) = \frac{\langle \varphi_{Z_t}(q_i) - q_i , p_i-q_i \rangle}{\lN  \varphi_{Z_t}(q_i) - q_i \rN_2 \lN p_i -q_i \rN_2}\] 
Note that since $\varphi_{Z_t}(q_i) - q_i = t \Delta Q e_i + t \Delta \mu$, the quantity $\alpha_i$ is independent of $t$. Now set:
\begin{equation}\label{eq:pf:tau}
\tau_i = \frac{2\cos(\alpha_i) \lN p_i - q_i \rN_2}{\lN \Delta Q e_i + \Delta \mu \rN_2} 
\end{equation}
Since $\cos(\alpha_i) > 0$ by assumption, we have $\tau_i > 0$ for all $i$. Set $\tau = \min(\epsilon_0,\min_i \tau_i)$. Then for all $t < \tau$ we have:
\begin{align*}
2 \cos(\alpha_i) \lN p_i - q_i \rN_2 & > t \lN \Delta Q e_i + \Delta \mu \rN_2 \\
& = \lN \varphi_{Z_t}(q_i) - q_i \rN_2 \quad \forall i
\end{align*}
From Figure \ref{fig:pf:tau}, we see this implies that 
\begin{equation}\label{eq:pf:optim}
\lN \varphi_{Z_t}(q_i) - p_i \rN_2 < \lN q_i - p_i \rN_2 = d_P(Z) \quad \forall i, \forall t < \tau
\end{equation}

\begin{figure}
\begin{center}
\begin{tikzpicture}
\draw[dashed, thick] (0,0) -- (0,3);
\draw[dashed, thick] (-2.4,1.18) -- (0,3);
\draw[dashed, red, thick] (0,0) -- (-2.4,1.18);
\draw[dashed, thick]  (-2.4,1.18) -- (-3,1.5);
\draw[fill=black] (0,3) circle (2pt) node[anchor=south] {$p_i$};
\draw[fill=black] (0,0) circle (2pt) node[anchor=north] {$q_i$};
\draw[fill=black] (-1,0.5) circle (2pt) node[anchor=north east] {$\varphi_{Z_t}(q_i)$};
\node[anchor=south east] at (0.1,0) {$\alpha_i$};
\node[rotate=0] at (0,1.5) {$=$};
\node[rotate=-45] at (-1.2,2.1) {$=$};
\end{tikzpicture}
\end{center}
\caption{Schematic illustration of $\tau_i$. For any $t<\tau_i$,  $\varphi_{Z_t}(q_i)$ lies in the red region, in which $\lN p_i - \varphi_{Z_t}(q_i) \rN_2 < \lN p_i - q_i \rN_2$.}
\label{fig:pf:tau}
\end{figure}
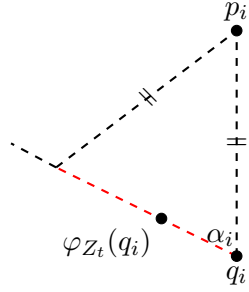

For $p \in \VERT(P)$ and $q \in \VERT(Z)$, let $F_p$ and $F_q$ be defined as in Theorem \ref{thm:local}. If $p_i$ is a vertex of $P$, then since $\varphi_{Z_t}(q_i) \in \varphi_{Z_t}(F_{p_i}) \subset \Aff(\varphi_{Z_t}(F_{p_i}))$, we have the inequality:
\[ d(p_i, \Aff(\varphi_{Z_t}(F_{p_i}))) \leq \lN p_i - \varphi_{Z_t}(q_i) \rN_2 \]
Similarly, if $q_j$ is a vertex of $Z$, then since $p_j \in F_{q_j}$ we have:
\[ d(\varphi_{Z_t}(q_j), \Aff(F_{q_j})) \leq \lN p_j - \varphi_{Z_t}(q_j) \rN_2 \]
Combining these with Equation \ref{eq:pf:optim}, we see that all terms achieving the maximum in Equation \ref{eq:localdp} are strictly smaller at $Z_t$ than at $Z$ for all $t< \tau$.

To prove claim 2), we leverage continuity. Let $V_P = \{ v \in \VERT(P) : v \neq p_i \ \forall i \}$ and $V_Z = \{ v \in \VERT(Z) : v \neq q_i \ \forall i\}$ be the sets of vertices not achieving the Hausdorff distance at $Z$. Choose $\gamma > 0$ such that $d(p,\Aff(F_p)) < d_P(Z) - \gamma$ for all $p$ in $V_P$ and $d(q,\Aff(F_q)) < d_P(Z) - \gamma$ for all $q \in V_Z$. Now define $f: [0,\tau) \to \R$ by:
\[ f(t) = \max\left(\max_{p\in V_P}d(p, \Aff(\varphi_{Z_t}(F_p))), \max_{q \in V_Z} d(\varphi_{Z_t}(q),\Aff(F_q)) \right) \]
This is a continuous function with $f(0) < d_Z(P) - \gamma$. By continuity, there exists $\tau' > 0$ such that $f(t) < d_Z(P)$ for all $t \in [0,\tau')$. Setting $\epsilon := \min(\tau,\tau')$, we see that both 1) and 2) hold. 
\end{proof}

\begin{rem}\label{rem:perturb}
The maximal quantity $\epsilon$ that we have constructed in the above proof we will refer to as the \emph{perturbation limit} of $Z$ relative to $P$. Note that any perturbation $Z_t$ with $t < \epsilon$ is, by construction, contained in the neighborhood $B_{\epsilon_0}(Z)$. This means that iteratively applying such perturbations $Z_0 \to Z_{t_1} \to Z_{t_2} \to ...$ will never change the functional form of $d_P$ at $Z_{t_i}$. In other words, $d_P(Z_{t_i})$ will always be a maximum of the same set of functions.
\end{rem}

This result shows that we have a checkable, sufficient condition for generic $Z$ not being a local minimum of $d_P$. In the next section, we will show how to explicitly check this condition and how this can be integrated into an optimization scheme for $d_P$.

\begin{exmp}
Here we calculate the feasibility cone for a simple case. Let $Z = Z(I,0)$, where $I \in \R^{2\times 2}$ is the identity matrix, and let $P$ be the polytope obtained by rotating $Z$ by $\frac{\pi}{4}$ and scaling about its barycenter by $1+\epsilon$ for a small quantity $\epsilon$ (see Figure \ref{fig:localdp}). In this case, there are four points $(p_0,q_0),...,(p_3,q_3)$ achieving $d_P(Z)$. The feasible cone $\mathcal{C}(P,Z)$ is the set $\{ z= (\Delta Q, \Delta \mu) \in \R^{2\times 2} \times \R^2 \mid Az \geq 0 \}$, where $A$ is:
\begin{align*}
A &= \begin{pmatrix} (p_0-q_0) \otimes (q_0,1) \\  (p_1-q_1) \otimes (q_1, 1) \\ (p_2-q_2) \otimes (q_2, 1) \\ (p_3-q_3)\otimes (q_3,1) \end{pmatrix}  \\
& = \begin{pmatrix} 0 & -1 & -2 & 0 & 0 & 0 \\ 0 & 0 & 0 & -1 & 0 & -2 \\ 2 & 1 & 2 & 0 & 0 & 0 \\ 0 & 0 & 0 & 1 & 2 & 2 \end{pmatrix}
\end{align*}
Here we are flattening $\Delta Q$ into a vector of length 4 for simplicity. The extremal rays of $\mathcal{C}(P,Z)$ can be computed\footnote{using Polymake} to be:
\[ 
\begin{matrix} v_1 = \begin{pmatrix} 1 & 0 & 0 & 0 & 0 & 0  \end{pmatrix} \\
v_2 = \begin{pmatrix} 0 & 0 & 0 & 0 & 1 & 0  \end{pmatrix} \\
v_3 = \begin{pmatrix} 0 & 0 & 0 & 0 & 1 & -1  \end{pmatrix} \\
v_4 = \begin{pmatrix} 1 & -2 & 0 & 0 & 1 & -1  \end{pmatrix} \\
\end{matrix}
\]
with lineality space generated by:
\[ 
\begin{matrix} 
\ell_1 = \begin{pmatrix} 0 & -2 & 1 & 0 & 0 & 0  \end{pmatrix} \\
\ell_2 = \begin{pmatrix} 0 & 0 & 0 & -2 & 0 & 1  \end{pmatrix} \\
\end{matrix}
\]
Any combination $\lambda_1 v_1 + \lambda_2 v_2 + \lambda_3 v_3 + \lambda_4 v_4 + \mu_1 \ell_1 + \mu_2 \ell_2$ with $\lambda_i > 0$ and $\mu_i \in \R$ lies in the interior of this cone. Therefore $Z$ is not a local minimum of $d_P$.

\end{exmp}

It is important to note that the converse of Proposition \ref{prop:optim-cone} doesn't necessarily hold. That is, if $\mathcal{C}(P,Z)$ has empty interior, it isn't necessarily true that any perturbation of $Z$ will not decrease the Hausdorff distance. The trouble happens when one of the $q_i$ is not a vertex of $Z$. If $\mathcal{C}(P,Z)$ has empty interior, that means for every $(\Delta Q, \delta \mu)$, the corresponding $\alpha_i$ satisfies $\cos(\alpha_i) \leq 0$. In this case just because $\lN p_i - \varphi_{Z_t}(q_i) \rN_2 > \lN p_i - q_i \rN_2$ doesn't mean that $d(p_i, \Aff(\varphi_{Z_t}(F_{p_i})))$ increases (see diagram below, and compare to Figure \ref{fig:pf:tau}).

\begin{center}
\begin{tikzpicture}
\draw[thick] (-1,0) -- (1,0);
\draw[dashed, thick] (0,0) -- (0,1);
\draw[red, thick] (1,0.5) -- (-2,-1);
\draw[fill=black] (0,1) circle (2pt) node[anchor=south] {$p_i$};
\draw[fill=black] (0,0) circle (2pt) node[anchor=north] {$q_i$};
\draw[fill=black] (-1,-0.5) circle (2pt) node[anchor=north] {$\varphi_{Z_t}(q_i)$};
\end{tikzpicture}
\end{center}
Here $F_{p_i}$ is the solid black line and $\Aff(\varphi_{Z_t}(F_p)$ is the solid red line. We see that $d(p_i, \Aff(\varphi_{Z_t}(F_{p_i})))$ still decreases even though the pushed forward point $\varphi_{Z_t}(q_i)$ gets further away from $p_i$. This difficulty occurs because $F_{p_i}$ is a facet of dimension at least 1. This leads us to the following

\begin{cor}\label{cor:optim-cone}
Given the same setup as Proposition \ref{prop:optim-cone}, suppose that each $q_i$ is a vertex of $Z$ (equivalently, each $F_{p_i}$ is a dimension 0 facet). Then $\mathcal{C}(P,Z)$ has empty interior if and only if $Z$ is a local minimum.
\end{cor}
\begin{proof}
We have already shown one direction in Proposition \ref{prop:optim-cone}. For the other direction, assume that $\mathcal{C}(P,Z)$ has empty interior. This means that any $(\Delta Q, \Delta \mu) \in \R^{n \times d} \times \R^d$ satisfies $\langle \Delta Q e_i + \Delta \mu, p_i - q_i \rangle \leq 0$ for all $i$. Note that for any $i$, we have that $p_i - q_i$ is orthogonal to $F_{q_i}$ because the Hausdorff distance is using the Euclidean metric. This means that for any perturbation $Z_t = Z(Q + t \Delta Q, \mu + t \Delta \mu)$, the pushforward $\varphi_{Z_t}(q_i)$ will not be closer to $F_{q_i}$ than $q_i$ because $\langle \varphi_{Z_t}(q_i) - q_i, p_i - q_i \rangle = \langle \Delta Q e_i + \Delta \mu , p_i - q_i \rangle \leq 0$ (see Figure \ref{fig:pf:optim-cone-cor}). Therefore all of the terms achieving the maximum in the local expression for $d_P$ will not decrease for any $t$ and for any choice of $(\Delta Q, \Delta \mu)$. Since $d_P$ is a maximum of these terms, this implies that $d_P$ will not decrease either. This shows that $Z$ is a local minimum.
\end{proof}

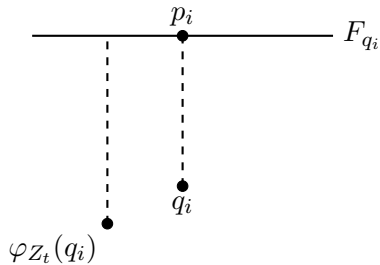
\begin{figure}
\begin{center}
\begin{tikzpicture}
\draw[thick] (-2,2) -- (2,2);
\draw[dashed, thick] (0,0) -- (0,2);
\draw[dashed, thick] (-1,-0.5) -- (-1,2);
\draw[fill=black] (0,2) circle (2pt) node[anchor=south] {$p_i$};
\draw[fill=black] (0,0) circle (2pt) node[anchor=north] {$q_i$};
\draw[fill=black] (-1,-0.5) circle (2pt) node[anchor=north east] {$\varphi_{Z_t}(q_i)$};
\node[anchor=west] at (2,2) {$F_{q_i}$};
\end{tikzpicture}
\end{center}
\caption{When $\langle \varphi_{Z_t}(q_i) - q_i , p_i - q_i \rangle \leq 0$, the distance from $\varphi_{Z_t}(q_i)$ to $F_{q_i}$ is greater than or equal to the distance from $q_i$ to $F_{q_i}$ because $(q_i - p_i) \perp F_{q_i}$.}
\label{fig:pf:optim-cone-cor}
\end{figure}

As a final note, we note an implication that this corollary has on the coarse Hausdorff distance $d_P^c$. Recall $d_P^c(Z)$ is the Hausdorff distance restricted to only the vertices of $P$ and $Z$ (Equation \ref{eq:dc}). Since all pairs $(p,q) \in P \times Z$ achieving $d_P^c(Z)$ are by definition vertices, the assumptions in the above corollary always apply and therefore we have a necessary and sufficient condition for all local minima of $d_P^c$; namely that $\mathcal{C}(P,Z)^\circ$ is empty. This shows that $d_p^c$ can have local minima which are not necessarily local minima of $d_P$.

\section{Optimizing $d_P$}
\label{sec:alg}

In this section, we describe an optimization algorithm for $d_P$ based on the subgradient method that leverages the feasibility cone $\mathcal{C}(P,Z)$. We use this cone to choose a subgradient of $d_P$ that can lead to a guaranteed decrease in $d_P$ for sufficiently small step sizes. Using the ideas from the proof of Proposition \ref{prop:optim-cone}, we will also establish three techniques for choosing a step size that balance convergence rate and monotonicity of the subgradient descent.

\subsection{Subdifferentials and the Subgradient Method}

We will first briefly remind the reader of subdifferentials and the subgradient method. For convex, continuous, possibly non-smooth functions $f: \R^d \to \R$, subgradient descent is a well-understood optimization technique inspired by standard gradient descent. It consists of generating a sequence $\{x_k\}_{k=0}^\infty \subset \R^d$ using the rule
\begin{equation}\label{eq:subgrad-method}
x_{k+1} = x_{k} - h_{k}(x_k) \nabla f(x_k) 
\end{equation}
where $\nabla f(x_k)$ is a \emph{subgradient} of $f$ at $x_k$ and $h_k(x_k)$ is a step size. 
\begin{defn}
For a convex function $f: U \to \R$ defined on an open set $U \subset \R^d$, a vector $v \in \R^d$ is a \emph{subgradient} of $f$ at $x_0$ if for all $x \in U$:
\[ f(x) - f(x_0) \geq \langle v, x-x_0 \rangle \]
\end{defn}
The collection of all subgradients of $f$ at $x_0$ is called the subdifferential $\del f(x_0)$. When $f$ is differentiable at $x_0$, the subdifferential $\del f(x_0)$ is a singleton set consisting of the gradient of $f$ at $x_0$.  Under mild conditions on $f$ and the step sizes, the subgradient method (\ref{eq:subgrad-method}) converges to the global minimum of $f$ \cite{Shor85}. It also has the useful property that $0 \in \del f(x_0)$ if and only if $x_0$ is the global minimum of $f$.

\begin{exmp}
Let $f: \R \to \R$ be the function $x \mapsto |x|$. Then it is a straightforward calculation to show that
\[ \del f(x) = \left\{ \begin{matrix} \{-1\} & x<0 \\ [-1,1] & x = 0 \\ \{1\} & x > 0 \end{matrix} \right.\]
Indeed we see that the only point for which $\del f(x)$ contains $0$ is the global minimum $x = 0$. 

Suppose instead $f(x) = \max(0,x-1) + \min(0, x+1)$ (see Figure \ref{fig:ex:nonconvex}). This is no longer convex, and we see
\[ \del f(x) = \left\{ \begin{matrix} \emptyset & x < 1 \\ \{ 1 \} & x \geq 1 \end{matrix} \right. \]
\end{exmp}

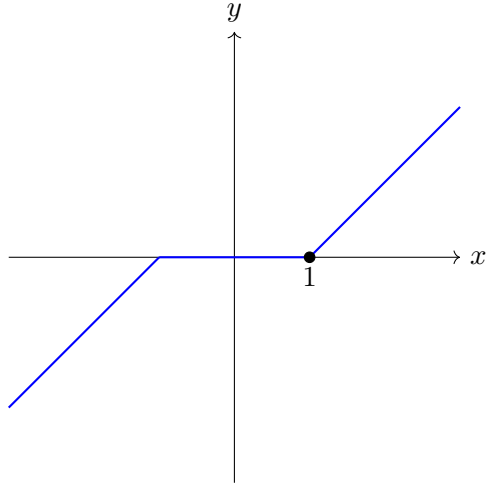
\begin{figure}
\begin{center}
\begin{tikzpicture}
  \draw[->] (-3, 0) -- (3, 0) node[right] {$x$};
  \draw[->] (0, -3) -- (0, 3) node[above] {$y$};
  \draw[scale=1, domain=-1:1, smooth, variable=\x, blue, thick] plot ({\x}, {0});
  \draw[scale=1, domain=-3:-1, smooth, variable=\x, blue, thick] plot ({\x}, {\x+1});
  \draw[scale=1, domain=1:3, smooth, variable=\x, blue, thick] plot ({\x}, {\x-1});
  \draw[fill=black] (1,0) circle (2pt) node[anchor=north]{1};
\end{tikzpicture}
\end{center}
\caption{A non-convex function with empty subdifferentials for $x < 1$.}
\label{fig:ex:nonconvex}
\end{figure}

As we saw in the above examples, when $f$ is not convex the subdifferential can be empty. In order to extend the subgradient method to non-convex functions, there are various localized versions of the subdifferential (see, for example \cite{LMK20}). One such version is the \emph{Clarke subdifferential}.
\begin{defn}
Let $f: \R^d \to \R$ be locally Lipchitz. The \emph{Clarke subdifferential} is:
\[ \del_Cf(x_0) = \Conv\left\{ v \in \R^d \mid \exists x_k \to x_0 \ \text{s.t.} \ \nabla f(x_k) \ \text{exists and } \nabla f(x_k) \to v \right \} \]
\end{defn}
This is a closed polyhedral set, and the Lipschitz property guarantees that it is nonempty for every $x$ in the domain of $f$ \cite{Clarke90}. Moreover, this coincides with the subdifferential $\del f(x)$ when $f$ is convex. The Clarke subdifferential also shares the following property to the usual subdifferential.

\begin{prop}[\cite{Clarke90}]
If $x_0$ is a local minimum for $f: \R^d \to \R$, then $0 \in \del_Cf(x_0)$.
\end{prop}
It is important to note that the converse doesn't necessarily hold in general. For a more detailed treatment of the Clarke subdifferential and its relationship to optimality, consult \cite{Clarke90,LMK20,RW98}.

\begin{exmp}
Returning to the above example, the Clarke subdifferential of $f(x) = \max(0,x-1) + \min(0,x+1)$ is easily seen to be:
\[ \del_Cf(x) = \left\{ \begin{matrix} \{ -1\} & x < -1 \\ [-1,0] & x = -1 \\ \{ 0 \} & -1 < x < 1 \\ [0,1] & x = 1 \\ \{ 1\} & x> 1 \end{matrix} \right.\]
Indeed, $0\in \del_Cf(x)$ for $x \in [-1,1]$ but none of these points are local minima.
\end{exmp}

While there is much known about the subgradient method when $f$ is convex (see \cite{RW98}), there are much fewer such guarantees when $f$ is non-convex. The primary difficulty is that, in general, the step sizes and choice of subgradient are not easily ascertained to lead to a decrease in the objective. Many modern algorithms exist that assume extra structure on $f$ and/or its subgradients, such as bundle methods, smoothing, and gradient sampling (\cite{BT21, Kiw86, BJ13}). 


In the remainder of this section, we describe how the feasibility cone $\mathcal{C}(P,Z)$ can be leveraged to determine good subgradients and step sizes for optimizing $d_P$. 

\subsection{Direction finding}
We showed in Proposition \ref{prop:optim-cone} that any point in the interior of $\mathcal{C}(P,Z)$ gives an update step to $Z$ that decreases $d_P(Z)$, provided the step size is less than some $\epsilon$. This motivates us to consider only the subgradients that lie in the interior of this cone as potential directions for subgradient descent.

\begin{defn}
Let $P$ be a fixed polytope and $Z \in \ZZ_n(\R^d)$ a zonotope, and suppose that they satisfy the locality conditions 1) and 2). Then we define the \emph{feasible subdifferential} of $d_P$ to be:
\[ \FF(P,Z) = \left\{ \begin{matrix} \mathcal{C}(P,Z)^\circ \cap \del_C d_P(Z) & \text{if $\mathcal{C}(P,Z)$ has nonempty interior} \\ \del_C d_P(Z) & \text{otherwise} \end{matrix} \right. \]
where $\del_C d_P$ is the Clarke subdifferential of $d_P$.
\end{defn}
Note that since both $\del_C d_P(Z)$ and $\mathcal{C}(P,Z)$ are polyhedral sets, the closure of the feasible subdifferential is also a polyhedral set. 

Our choice of subgradient for optimizing $d_P$ will be an element of $\FF(P,Z)$. Recall from the proof of Proposition \ref{prop:optim-cone} that the closer a point $(\Delta Q, \Delta \mu) \in \mathcal{C}(P,Z)$ is to the boundary, the smaller the value of the perturbation limit $\epsilon$ that we constructed. Therefore, points that are further from the boundary will have higher perturbation limits, and hence can give larger step sizes. Motivated by this property, our choice of subgradient will be the \emph{Chebyshev center} (see remark below) of the closure of $\FF(P,Z)$. It is also possible that $\FF(P,Z)$ can be empty. In this case, we interpret $Z$ as being a local minimum, and so no further step direction is necessary.

\begin{rem}
Recall that the Chebyshev center of a convex polytope $X\subset \R^d$ is:
\[ X_{\text{Ch}} = \argmin_{c \in \R^d} \max_{x\in X} \lN x- c \rN_2 \]
It is the center of the maximal-radius ball entirely contained in $X$.
\end{rem}

To compute $\overline{\FF(P,Z)}_{\text{Ch}}$, we first need to determine if $\mathcal{C}(P,Z)$ has empty interior. To do this, one can compute the v-representation of this cone and check that the extremal rays form a spanning set of $\R^d$. We also need to be able to compute $\del_C d_P(Z)$. We give a detailed account of this computation in the next subsection.

Finally, the Chebyshev center of $\overline{\FF(P,Z)}$ is found by extracting its h-representation $\overline{\FF(P,Z)} = \{ x \mid \langle a_i , x \rangle \leq b_i, i = 1,...,m\}$ and solving the following linear program for $x_c$ and $r$:
\[ \begin{array}{l} \text{maximize } r \\ \text{s.t.} \ \langle a_i, x_c \rangle + r \lN a_i \rN \leq b_i, \ i = 1,...,m \end{array} \]
For more details about the Chebyshev center, see \cite{BL04}. Finally, we remark that, since $d_P$ is differentiable almost everywhere, the Clarke subdifferential $\del_Cd_P$ is a single point almost everywhere. Therefore the feasible subdifferential $\mathcal{F}(P,Z)$ is often just a single point, in which cases computing the Chebyshev center is unnecessary.

\subsection{Computing $\del_C d_P(Z)$}\label{subsec:computation}

We have already shown that $d_P$ is a locally Lipschitz function and that it can be locally written as a pointwise maximum of differentiable functions in a neighborhood of a sufficiently general point $Z \in \ZZ_n(\R^d)$. This information allows us to use the following formula for $\del_C$:

\begin{prop}
Let $U\subset \R^d$ be open and $f: U \to \R$ be locally Lipschitz and suppose $f(x) = \max(f_1(x),...,f_m(x))$, where $f_i: U \to \R$ are continuously differentiable on $U$. Then for any $x \in U$ the Clarke subdifferential is:
\[ \del_C f(x) = \Conv(\{\nabla f_i(x) \mid f_i(x) = f(x) \}) \]
\end{prop}
\begin{proof}
Any function $f_i$ achieving the maximum at $x$ has the property that there exists $x_k \to x$ such that $f(x_k) = f_i(x)$ and $f(x_k) < f_j(x_k)$ for $j\neq i$. Since $\nabla f_i$ is continuously differentiable, $\nabla f_i(x_k) \to \nabla f_i(x)$. Moreover $\nabla f_i(x_k) = \nabla f(x_k)$ because $f_i$ uniquely achieves the maximum at $x_k$. Therefore $\nabla f_i(x)$ is a limiting gradient. Every limiting gradient is of this form, since wherever $\nabla f$ is defined, it is equal to some $\nabla f_i$. Thus $\del_C f(x)$, which is the convex hull of the limiting gradients, is as claimed.
\end{proof}

As a warm-up, we first use this to compute $\del_C d^c_P(Z)$, which is the Clarke subdifferential of the coarse Hausdorff distance $d_P^c$. Recall this is the Hausdorff distance between the vertex sets of $P$ and $Z$, which is an upper bound for $d_P$ and is simpler to compute.

\begin{cor}\label{cor:coarse-grads}
Let $Z \in \ZZ_n(\R^d)$ be a general position zonotope and let $(p_1,q_1),...,(p_k,q_k)$ be the pairs of vertices where $d_P^c(Z)$ is achieved. Let $e_i = \Lift(q_i)$ and $Q,\mu$ be such that $Z = Z(Q,\mu)$. Then:
\[ \del_C d_P^c(Z) = \Conv\left( \{ \nabla (\lN p_i - Qe_i - \mu \rN_2) \} \right) \]
where the gradient is taken in terms of the parameters $Q,\mu$.
\end{cor}
\begin{proof}
Observe that $d_P^c(Z) = \max_i (\lN p_i - q_i \rN_2)$ and that $q_i = Qe_i +\mu$. Then apply the above Proposition.
\end{proof}

Using Theorem \ref{thm:local}, we also arrive at the analogous formula for $\del_C d_P(Z)$:

\begin{cor}
Let $P$ and $Z_0 \in \ZZ_n(\R^d)$ satisfy the locality conditions 1) and 2), and let $(p_1,q_1),...,(p_k,q_k)$ be the pairs of points where $d_P(Z_0)$ is achieved. Further, let $U$ be the neighborhood from Theorem \ref{thm:local} under which $d_P|_U$ is a maximum of distances between points and affine subspaces. Without loss of generality, assume that $q_1,...,q_\ell \in \VERT(Z_0)$ and $q_{\ell+1},...,q_k \notin \VERT(Z_0)$. Then for all $Z \in U$:
\[ \del_C d_P(Z) = \Conv\left(\{ \nabla \delta_{q_i}(Z) \}_{i=1,...,\ell} \cup \{ \nabla \delta^{p_j}(Z)\}_{i=\ell+1,...,k} \right) \]
Where for $q \in \VERT(Z_0)$, the function $\delta_q$ is $Z \mapsto d(\varphi_Z(q), \Aff(F_{q}))$ and for $p \in \VERT(P)$, the function $\delta^p$ is $Z \mapsto d(p, \Aff(\varphi_Z(F_p)))$. For definitions of $\varphi_Z, \delta, F_p$ and $F_q$, refer to Theorem \ref{thm:local}.
\end{cor}

We have thus reduced the problem of computing the Clarke subgradient of $d_P$ to computing gradients of the functions $\delta_q$ and $\delta^p$, which are both distances between a point and an affine subspace. Since these vary by how they depend on $Z$, we have to treat them separately. 

Both $\delta_q$ and $\delta^p$ are distances between a point and an affine subset, and so it will be useful for us to first establish an explicit formula for such a quantity. Let $W$ be an affine subspace and $u$ be a point. The distance between $W$ and $u$ is defined to be the distance between $u$ and the projection of $u$ onto $W$:
\[ d(u,W) = \lN u - \Pi(u,W) \rN_{2} \]
For this to be useful to us, we will need to find a more explicit expression for it when $W$ is given as an intersection of codimension 1 affine subsets.
\begin{lem}\label{lem:affine-proj}
Let $W \subset \R^d$ be an affine subspace and $a = \Pi(0,W)$, so that $W' = W - a$ is a subspace. Let $W^\perp$ be the orthogonal complement of $W'$. Then for any $u \in \R^d$, we have $u-\Pi(u,W) \in W^\perp$.
\end{lem}
\begin{proof}
Now, the orthogonal decomposition $\R^d = W' \oplus W^\perp$ gives us a unique expression $u = \Pi(u,W') + y$ where $y \in W^\perp$. Therefore we also have $u = \Pi(u,W')+a + (y-a)$. Since $\Pi$ is linear in both arguments, we have $\Pi(iu,W') = \Pi(u,W-a) = \Pi(u,W) - a$. Thus $u = \Pi(u,W) + (y-a)$. Finally note that $a \in W^\perp$, and so $y-a \in W^\perp$ as well.
\end{proof}
Suppose that $W$ is given to us as an intersection of codimension 1 affine subspaces $W = \bigcap_i^m H_i$, where each $H_i = \{ y \in \R^d \mid \langle \eta_i,y \rangle = c_i \}$ with $\eta_i$ are unit norm. We would like to get an expression for $d(u,W)$ in terms of the $\eta_i$ and $c_i$ quantities. In this case, $W^\perp = \Span(\eta_i)$, c.f. Lemma \ref{lem:affine-proj}. Letting $x := \Pi(u,W)$, by Lemma \ref{lem:affine-proj}, $u - x \in W^\perp$, so we can expand it in the $\eta_i$ basis. 
\[ u - x= \sum_{i}^m \langle \eta_i, u-x \rangle \eta_i = \sum_{i=1}^m \langle \eta_i, u \rangle \eta_i - \sum_{i=1}^m \langle \eta_i ,x \rangle \eta_i\]
Since $x \in W$, we get $\langle \eta_i, x\rangle = c_i$. Thus
\[ u-x = \sum_{i=1}^m (\langle \eta_i, u \rangle - c_i ) \eta_i \]
We now can compute the norm:
\begin{equation}\label{eq:delta-explicit}
d(u,W) = \lN u-x \rN_2 = \left\lN  \sum_{i=1}^m (\langle \eta_i, u \rangle - c_i ) \eta_i \right\rN_2\end{equation}


\begin{prop}\label{prop:grad1}
Let $P$ and $Z_0 \in \ZZ_n(\R^d)$ satisfy locality conditions 1) and 2) and let $U$ be the neighborhood of $Z_0$ obtained from Theorem \ref{thm:local}. Fix any $q \in \VERT(Z_0)$, set $e = \Lift(q)$, and assume that $W = \Aff(F_q)$ (c.f. Theorem \ref{thm:local}) is codimension one, so that we can write $W =\{ y \in \R^d \mid \langle \eta, y\rangle = c\}$. For $Z = Z(Q,\mu) \in U$, considering $\delta_q(Z) = d(\varphi_Z(q), W)$ as a function of the parameters $Q = (g_{ij})$ and $\mu = (\mu_j)$, for all $1 \leq i \leq n$ and $1 \leq j \leq d$ we have::
\[ \frac{\del \delta_q}{\del g_{ij}} (Z) =\eta_j e_i \]
\[ \frac{\del \delta_q}{\del \mu_j} (Z) = \eta_j\]
where $\mathds{1}(e_i=1)$ is the indicator function for when $e_i = 1$.
\end{prop}
\begin{proof}
Since $W = \Aff(F_q)$ is codimension one, the distance between $\varphi_Z(q)$ and $W$ is $d(\varphi_Z(q),W) = \langle \eta, \varphi_Z(q) \rangle - c$. Since $\varphi_Z(q) = Qe + \mu = \sum_{i' \mid e_{i'} = 1} g_{i'}$, we can substitute that and differentiate with respect to $g_{ij}$ and with respect to $\mu_j$, noting that $\eta$ and $c$ are independent of $Q$ and $\mu$.
\end{proof}

When $W = \Aff(F_q)$ is higher dimension, we can obtain the gradients $\frac{\del \delta_q}{\del g_{ij}}$ and $\frac{\del \delta_q}{\del \mu_j}$ in a similar way by substituting $u = Qe + \mu$ into equation \ref{eq:delta-explicit} and differentiating. We omit this calculation for brevity.

\begin{cor}\label{cor:grad1}
The function $\delta_q(Z)$ is the square root of a bounded rational function.
\end{cor}
\begin{proof}
Note that $Qe + \mu$ is a component-wise polynomial function of the parameters $(Q,\mu)$. The substitution of $u = Qe + \mu$ into Equation \ref{eq:delta-explicit} introduces a square root.
\end{proof}

Calculating the gradient of $\delta^p$ is a bit more involved. In this case, the parameters $\eta_i, c_i$ in Equation \ref{eq:delta-explicit} are dependent on the variables $Q$ and $\mu$, and $u$ is a constant; whereas above it was the opposite. To simplify the algebra, just as above, we first consider the case where $W$ is codimension 1 space (i.e. $m = 1$). The following characterization of facets of a zonotope will be necessary to carry out this calculation.

\begin{lem}
Let $Z \in \ZZ_n(\R^d)$ be a general position zonotope and let $F \subset Z$ be a facet (codimension 1 face) with outward pointing normal $\eta$. Then there exist precisely $d-1$ generators $g_1,...,g_d$ of $Z$ such that $\langle \eta, g_i \rangle  =0$.
\end{lem}
\begin{proof}
See \cite{McM70}.
\end{proof}

\begin{rem}
If $Z = Z(Q,\mu)$ and $F$ is a facet of $Z$, we denote $Q_F \in \R^{(d-1)\times d}$ to be the submatrix of $Q$ whose rows are the $d-1$ generators orthogonal to $\eta$. We denote $[Q_F]_{j}$ to be the determinant of the square submatrix of $Q_F$ obtained by omitting the $j$th column of $Q_F$.

\end{rem}

\begin{prop}\label{prop:grad2}
Let $P$ and $Z_0 \in \ZZ_n(\R^d)$ satisfy the locality conditions 1) and 2) and let $U$ be a neighborhood of $Z_0$ obtained from Theorem \ref{thm:local}. Assume that $p \in \VERT(P)$ is a vertex such that the projection of $p$ onto $Z_0$ lies on a facet (codimension one face) $F$ of $Z_0$. Additionally, fix some $v \in \VERT(Z_0) \cap F$, and write $e = \Lift(v)$. For $Z = Z(Q,\mu) \in U$, considering $\delta^p(Z)$ as a function of the parameters $Q = (g_{ij})$ and $\mu = (\mu_j)$, for all $1 \leq i \leq n$ and $1 \leq j \leq d$ we have:
\[ \frac{\del \delta^p}{\del g_{ij}}(Z) = - \eta_j e_i + \sum_{j'=1}^d \frac{\del \eta_{j'}}{\del g_{ij}} (p_{j'} - v_{j'}) \]
\[ \frac{\del \delta^p}{\del \mu_j}(Z) = - \eta_j  \]
where $\eta_{j} = \frac{(-1)^{j}\sigma}{\gamma}[Q_F]_{j}$, $\gamma = \sqrt{[Q_F]_1^2+ ... + [Q_F]_d^2}$ and $\sigma \in \{-1,1\}$ is an appropriate sign such that $\langle \eta, p - v \rangle> 0$. Moreover, this formula is independent of the choice of $v \in \VERT(Z_0) \cap F$.
\end{prop}
\begin{proof}
Since $F$ is a facet, $\delta^p$ is given by:
\[ \delta^p(Z) = \langle \eta, p \rangle - c \]
where $\eta$ is the unique unit normal to $F$ satisfying $\langle \eta, p-v \rangle > 0$, and $c$ is an appropriate scalar. Note that both $\eta$ and $c$ are functions of $Z$ (i.e. of $Q$ and $\mu$). Therefore:
\begin{equation}\label{eq:pf:eta}
\frac{\del \delta^p}{\del g_{ij}}(Z) = \sum_{j'=1}^d p_{j'} \frac{\del \eta_{j'}}{\del g_{ij}} - \frac{\del c}{\del g_{ij}}
\end{equation}
We note that $c = \langle \eta, v \rangle$ because $v \in F$. Since $v = Qe + \mu$, we can substitute and differentiate:
\[ \frac{\del c}{\del g_{ij}} = \frac{\del}{\del g_{ij}} \langle \eta, Qe + \mu \rangle = \sum_{j'=1}^d \frac{\del \eta_{j'}}{\del g_{ij}} v_{j'} + \eta_{j'} e_i \mathds{1}(j=j') \]
where $\mathds{1}(j=j') = 1$ if $j=j'$ and $0$ otherwise. Simplifying, we find:
\[ \frac{\del c}{\del g_{ij}} = -\eta_j e_i + \sum_{j'=1}^d \frac{\del \eta_{j'}}{\del g_{ij}} v_j \]
Putting this back into (\ref{eq:pf:eta}), we get the first claimed expression. Doing a similar differentiation process, but with respect to $\mu_j$ and noting that $\eta$ is independent of $\mu$, we get the second claimed expression.

We now must show that $\eta$ satisfies $\eta_j = \frac{(-1)^j\sigma}{\gamma} [Q_F]_j$. We note that any normal vector $\tilde{\eta}$ to $F$ satisfies:
\[ Q_F\tilde{\eta}  = 0 \]
The matrix $Q_F \in \R^{(d-1)\times d}$ is full rank and therefore has one-dimensional kernel generated by any such normal vector. Gaussian elimination on the system above gives $\tilde{\eta} = ( (-1)[Q_F]_1 ,...,(-1)^j[Q_F]_j,..., (-1)^d[Q_F]_d)$. Then choosing $\sigma \in \{1, -1\}$ so that $\langle \sigma \tilde{\eta}, p - v \rangle > 0$ and normalizing $\tilde{\eta}$, we get $\eta$ as we defined it.

Finally, we must show that this was independent of our choice of $v \in \VERT(Z_0) \cap F$. Recall in (\ref{eq:pf:eta}) that only dependence on $v$ is in the term $\frac{\del c}{\del g_{ij}}$, where $c = \langle \eta, v \rangle$. However, $c$ is invariant under our choice of $v$ because every $v \in F$ satisfies $\langle \eta, v \rangle = c$.
\end{proof}

The general case when $W$ is higher codimension can be performed by differentiating Equation \ref{eq:delta-explicit} and using the chain rule. We have already established how to differentiate the inside of the summation, since each of those is a codimension 1 case.

\begin{cor}\label{cor:grad2}
With the same setup as above, the function $\delta^p(Z)$ is the square root of a bounded rational function.
\end{cor}
\begin{proof}
This follows from our formula for $\eta_j$ in the proof above, which is the square root of a rational function of the parameters of $Q$. The value $\gamma = \sqrt{[Q_F]_1^2 + ... + [Q_F]_d^2}$ never vanishes because $Q$ is full rank by assumption.
\end{proof}

\begin{rem}\label{rem:hashtable}
Computationally, the quantities $\eta_j$ are the most expensive part of calculating the gradient of $\delta^p(Z)$ because they involve determinants. However, note that the determinental quantities $[Q_F]_j$ are minors of a matrix whose entries are all formal variables $g_{ij}$. This means that we can store all possible such minors ahead of time in a hashtable (lookup table) and then $\eta_j$ can be computed by evaluating these at the current values of $g_{ij}$. This hashtable will have ${ n \choose d-1}\cdot d$ entries since each $Q_F$ corresponds to a $d-1$ sized subset of $\{1,...,n\}$ and there are $d$ possible minors for $Q_F$.
\end{rem}


\subsection{Step sizes}
To choose an appropriate step size $h$, recall the perturbation limit $\epsilon$ from Proposition \ref{prop:optim-cone}. This was the minimum of two quantities, $\tau$ and $\tau'$. The former is explicitly given by (\ref{eq:pf:tau}):
\[ \tau = \min_i (\tau_i, \epsilon_0) \]
This represents the maximum step size that one can take to ensure that all terms achieving the maximum of $d_P(Z)$ decrease. While any step size $t < \tau$ leads to a decrease in the Hausdorff distance, the optimum step size is not necessarily $\tau$. This can be seen in Figure \ref{fig:pf:tau}, where it is clear that the optimal value of $t$ that minimizes $\lN \varphi_{Z_t}(q_i) - p_i\rN$ is $\frac{\tau_i}{2}$

With this in mind, we define three step size rules:

\begin{enumerate}
\item {\bf Conservative step:} Set $h_k = \frac{1}{2}\min_i(\tau_i)$. This ensures that $d_P(Z)$ decreases strictly.
\item {\bf Random step:} Set $h_k = \frac{1}{2} \text{RandomChoice}(\{\tau_1,...\tau_k\})$.
\item {\bf Aggressive step:} Set $h_k = \frac{1}{2} \max_i(\tau_i)$. This ensures that at least one of the pairs achieving the Hausdorff distance decrease in distance.
\end{enumerate}

The conservative step rule has the advantage of ensuring that $d_P(Z)$ decreases; however, it has the potential to be a very small step size, which can hinder convergence and cannot explore very far (recall Remark \ref{rem:perturb}). Conversely, the aggressive step gives a larger step size and better exploration, but might not always lead to a decrease in the Hausdorff distance. In between these is the random step method, which is a compromise. In practice, it is likely best to use a mixture of these step sizes, such as starting with aggressive steps at the beginning to explore the landscape and then switching to conservative steps after an amount of time to more reliably converge to a local minimum.

\subsection{The algorithm}

Our proposed algorithm for subgradient descent on $d_P$ is shown in Algorithm \ref{alg:subgradient}, which uses the random choice step size described above. We have also included a few practical considerations. First, recall that our subgradient calculus was only valid for zonotopes satisfying the locality conditions 1) and 2); if in the iteration loop $Z$ does not satisfy these, since these are generic conditions, we can perturb it by a small random quantity and it will satisfy them almost surely. Additionally, through experiments, we found it highly effective to start with a good initial guess (warmstart) for $Z$ rather than a randomly initialized starting point.

A good technique for warmstarting is to compute an approximate minimal enveloping zonotope $\overline{Z}$ of $P$. An enveloping zonotope of $P$ is a zonotope that entirely contains $P$. Work by \cite{GNZ03} showed how to compute an approximate minimal enveloping zonotope for any $n$ and $d$. For example, in the case of $d=2$, for any point $O \in \R^2$ they define:
\[ \overline{Z}_O = \Conv\left( \bigcup_{v \in \VERT(P)} \{ v, \REF_{O}(v)\} \right) \]
where $\REF_{O}$ is reflection about the point $O$. This is a centrally symmetric polygon containing $P$, and thus a zonotope. The authors show that the area of $\overline{Z}_O$ is a piecewise convex affine function whose global minimum can be found through a binary search in $\O(n^2\log n)$. For $d > 2$, they use a different technique based on solving a linear program, since $\overline{Z}_{O}$ isn't necessarily a zonotope in this case. They show that, given a tolerance $\epsilon > 0$, the enclosing zonotope whose sum of generators is within $\O(\epsilon)$ of the optimum can be computed in $\O(n \epsilon^{-(d-1)^2} + \epsilon^{-\O(d^2)})$.


\begin{algorithm}
\caption{Subgradient Method with Random Choice Stepping}
\label{alg:subgradient}
\begin{algorithmic}[1]
\Require A polytope $P \subset \R^d$, a threshold $\epsilon \geq 0$, an integer $n$, and a maximum number of steps $N$.
\Ensure A zonotope $Z$ of rank $n$ with small $d_P(Z)$.
\Procedure{OptimizeHausdorff}{$P,n,\epsilon,N$}
\State $Z \gets $ \textsc{Warmstart}$(P)$
\While{$d_P(Z) > \epsilon$ and number of iterations $\leq N$}
\While{$Z$ and $P$ do not satisfy the locality conditions 1) and 2)}
\State $Z \gets Z + \Delta Z$ with $\Delta Z \in \ZZ_n(\R^d)$ small and random.
\EndWhile
\State $(Q,\mu) \gets $ representative of $Z$ in $\widetilde{\ZZ_n}(\R^d)$
\If{$\FF(P,Z) = \emptyset$}
\State return $Z$
\EndIf
\State $(\Delta Q, \Delta \mu) \gets \overline{\FF(P,Z)}_{\text{Ch}}$
\State $\tau_i \gets$ Equation \ref{eq:pf:tau}
\State $h \gets \frac{1}{2} \text{RandomChoice}(\{\tau_i\})$
\State $Z \gets Z(Q+ h\Delta Q, \mu + h\Delta \mu)$
\EndWhile
\State return $Z$
\EndProcedure
\end{algorithmic}
\end{algorithm}

\section{Complexity Analysis and Numerical Results}
\label{sec:numerical}

The step-wise cost of Algorithm \ref{alg:subgradient} lies in the computation of the feasible subdifferential $\overline{\FF(P,Z)_{\text{Ch}}}$ and the Hausdorff distance $d_P$. Recall that since $d_P$ is differentiable almost everywhere, the feasible subdifferential will generically be a single point, given by one of the gradients computed in Propositions \ref{prop:grad1} and \ref{prop:grad2}. Therefore the average case complexity for a single step of Algorithm \ref{alg:subgradient} is determined by the complexity of computing the Hausdorff distance $d_P(Z)$ and the complexity of computing the gradient. The former turns out to be the dominating term if you are willing to precompute a potentially large hashtable:

\begin{prop}
Fix a polytope $P \subset \R^d$ and a zonotope rank $n\geq 1$. Then given a space complexity budget of $\O\left( d { n \choose d-1 } \right)$, the average case time complexity of a single step of Algorithm \ref{alg:subgradient} is the same as the time complexity of evaluating the Hausdorff distance between $P$ and a general position zonotope $Z \in \ZZ_n(\R^d)$.
\end{prop}

\begin{proof}
We can store all possible formal expressions for $\eta_j$ and $\frac{\del \eta_{j'}}{\del_{g_{ij}}}$ using the given storage budget (see Remark \ref{rem:hashtable}), and therefore the evaluation of the gradients of $\delta^p$ and $\delta_q$ can be performed in the time it takes to evaluate a rational function of degree $d-1$ at a point, which is $\O(d)$. Thus the computation of the feasible differential is generically $\O(d)$. The complexity of evaluating the Hausdorff distance $d_P(Z)$ is higher than $\O(d)$ because it involves solving a series of quadratic programs (see \cite{Kon14}). Thus the dominating term is the Hausdorff distance.
\end{proof}

We have implemented Algorithm \ref{alg:subgradient} in \texttt{python}, and that code can be found here:
\[\text{\url{https://github.com/geodavic/zonopt}.}\]
This implementation uses the \texttt{autograd} framework of \texttt{pytorch} to get the subgradients, and thus avoids the space complexity shown in the above Proposition. In practice this implementation is simpler than symbolic gradients and still performant.


In Figure \ref{fig:loss}, we see the result of some experiments where we generate a random polytope $P$ and report the distance $d_P(Z)$ as it is optimized. To illustrate the usefulness of warmstarting, we compared using randomly initialized zonotopes versus the warmstart method of computing an approximate enveloping zonotope of $P$. In addition, see Figure \ref{fig:eight-examples} for some example optimal zonotopes in dimension 2.


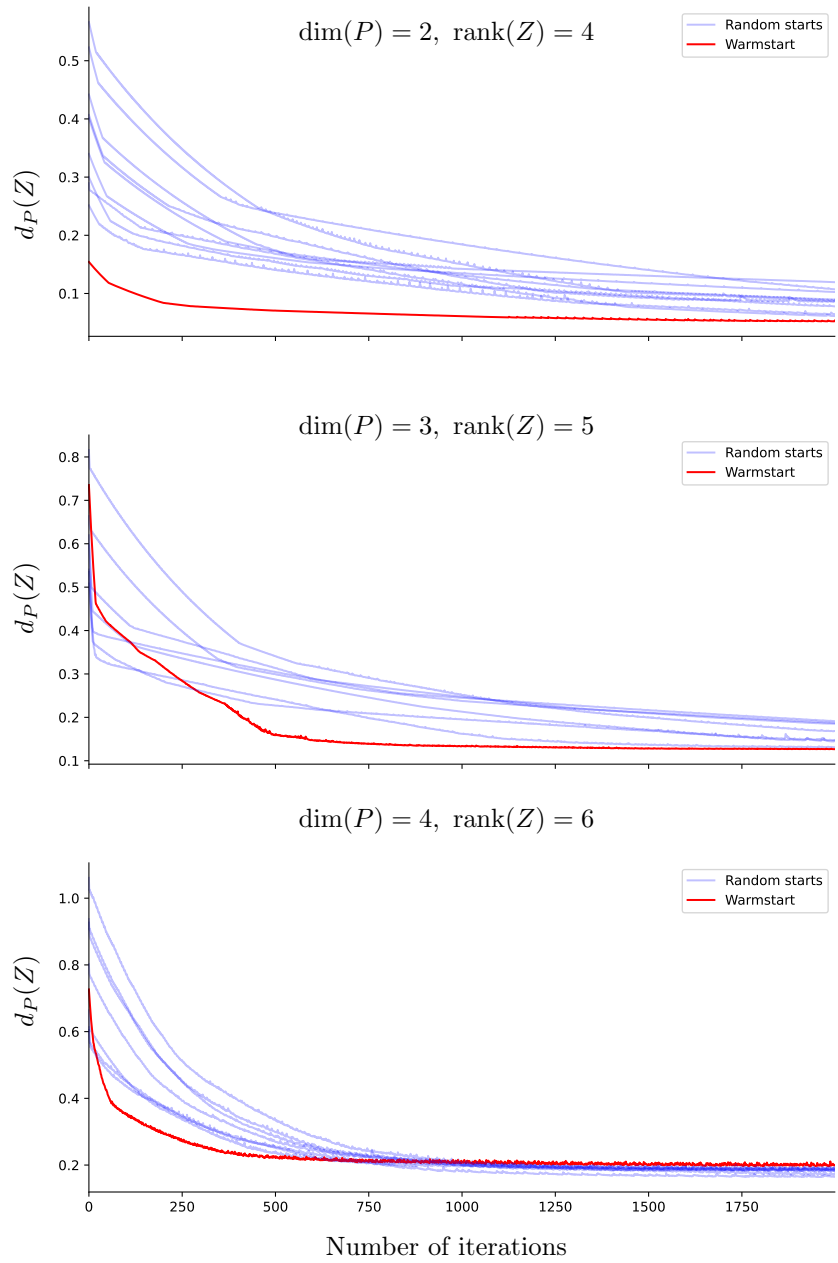
\begin{figure}
\begin{center}
\begin{tikzpicture}[scale=0.8725]
\node at (0,0) {\includegraphics[width=0.7\textwidth]{"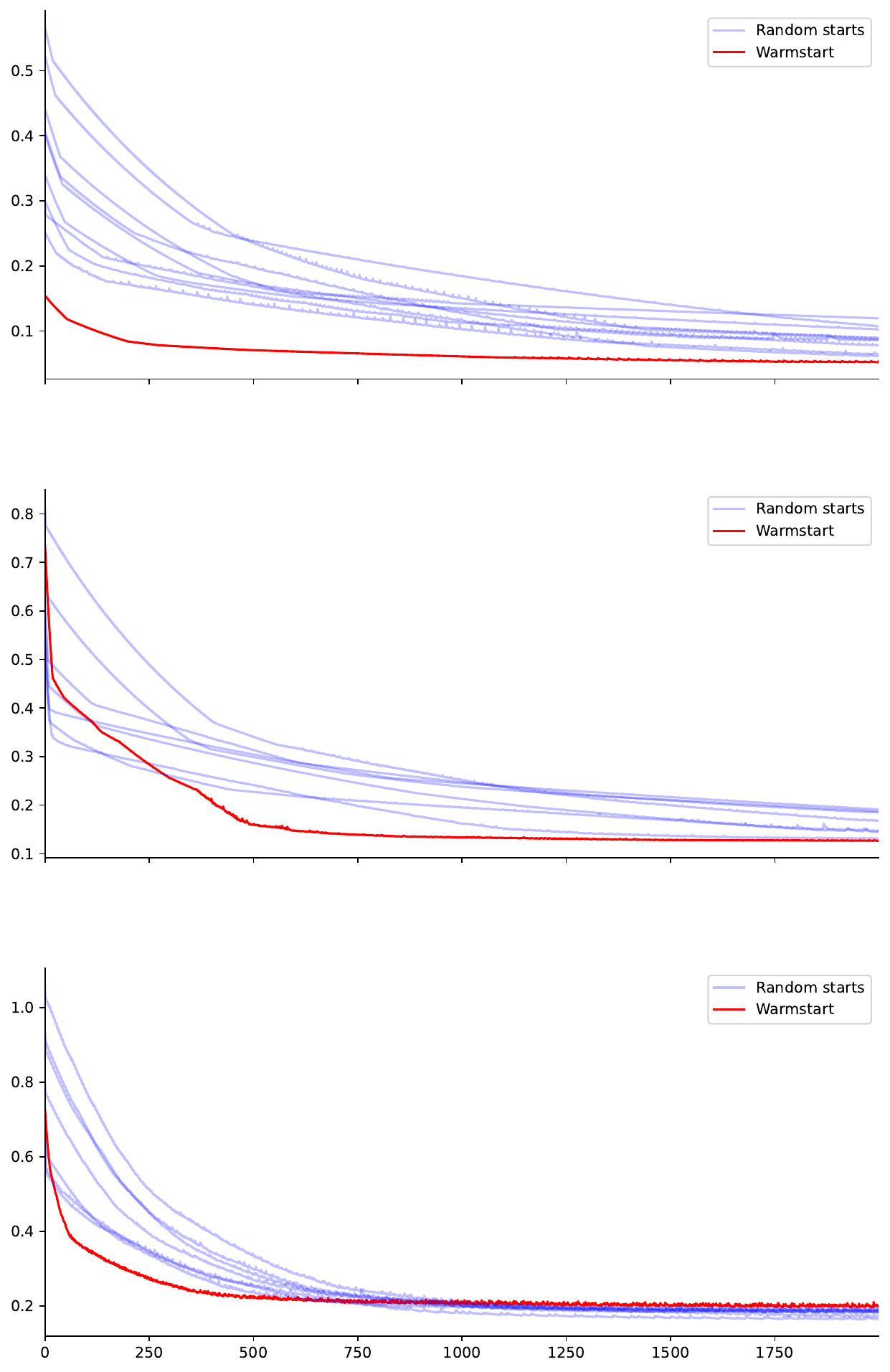"}};
\node at (0,8.8) {\small{$\dim(P) = 2, \  \rank(Z) = 4$}};
\node at (0,2.8) {\small{$\dim(P) = 3, \  \rank(Z) = 5$}};
\node at (0,-3.2) {\small{$\dim(P) = 4, \  \rank(Z) = 6$}};
\node[anchor=north] at (0,-9.4) {\small{Number of iterations}};
\node[rotate=90] at (-6.4,-5.8) {\small{$d_P(Z)$}};
\node at (6,-5.8) {$\quad$};
\node[rotate=90] at (-6.4,-5.8+6) {\small{$d_P(Z)$}};
\node at (6,-5.8+6) {$\quad$};
\node[rotate=90] at (-6.4,-5.8+12) {\small{$d_P(Z)$}};
\node at (6,-5.8+12) {$\quad$};
\end{tikzpicture}
\end{center}
\caption{Using Algorithm \ref{alg:subgradient} to approximate a random polytope $P$ by a zonotope $Z$ in various dimensions. The curves represent the Hausdorff distance $d_P(Z)$ at each step. The red curve corresponds to using a warmstarted guess for $Z$, and the rest correspond to random initializations of $Z$.}
\label{fig:loss}
\end{figure}

\begin{figure}
\begin{center}
\begin{tikzpicture}

\draw[thick,blue] (2.069361681022755, 0.3257679381477062) -- (3.0292995820151747, 2.940419925848613) -- (2.586469506592187, 3.1432990696186245) -- (0.30335567522522655, 2.6825343822841288) -- (0.9356551478793855, 0.36328597586710126) -- cycle;

\draw[thick,red] (0.48850021461868376, 0.8864197770002206) -- (0.6440321456221483, 0.3228069891923682) -- (1.781217336342005, 0.47700798034757586) -- (2.673877764127658, 1.120117137986416) -- (2.7624725735793207, 2.807449771467871) -- (2.6069406425758563, 3.3710625592757233) -- (1.4697554518559994, 3.2168615681205157) -- (0.5770950240703466, 2.5737524104816756) -- cycle;

\draw[opacity=0.8,gray,thick] (2.069361681022755, 0.3257679381477062) -- (1.8991769585220086, 0.5619909267113569);
\draw[fill=black] (2.069361681022755, 0.3257679381477062) circle (1.5pt);
\draw[fill=black] (1.8991769585220086, 0.5619909267113569) circle (1.5pt);
\draw[opacity=0.8,gray,thick] (3.0292995820151747, 2.940419925848613) -- (2.7472566948355706, 2.8625886987761153);
\draw[fill=black] (3.0292995820151747, 2.940419925848613) circle (1.5pt);
\draw[fill=black] (2.7472566948355706, 2.8625886987761153) circle (1.5pt);
\draw[opacity=0.8,gray,thick] (0.30335567522522655, 2.6825343822841288) -- (0.5770950240703465, 2.5737524104816747);
\draw[fill=black] (0.30335567522522655, 2.6825343822841288) circle (1.5pt);
\draw[fill=black] (0.5770950240703465, 2.5737524104816747) circle (1.5pt);
\draw[opacity=0.8,gray,thick] (0.771963368486245, 0.9637006763689252) -- (0.48850021461868376, 0.8864197770002206);
\draw[fill=black] (0.771963368486245, 0.9637006763689252) circle (1.5pt);
\draw[fill=black] (0.48850021461868376, 0.8864197770002206) circle (1.5pt);
\draw[opacity=0.8,gray,thick] (0.9257514171972658, 0.39961245184164634) -- (0.6440321456221483, 0.3228069891923682);
\draw[fill=black] (0.9257514171972658, 0.39961245184164634) circle (1.5pt);
\draw[fill=black] (0.6440321456221483, 0.3228069891923682) circle (1.5pt);
\draw[opacity=0.8,gray,thick] (2.3981612903833023, 1.2213431115844997) -- (2.673877764127658, 1.120117137986416);
\draw[fill=black] (2.3981612903833023, 1.2213431115844997) circle (1.5pt);
\draw[fill=black] (2.673877764127658, 1.120117137986416) circle (1.5pt);
\draw[opacity=0.8,gray,thick] (1.5277230402762463, 2.9296290467680826) -- (1.4697554518559994, 3.2168615681205157);
\draw[fill=black] (1.5277230402762463, 2.9296290467680826) circle (1.5pt);
\draw[fill=black] (1.4697554518559994, 3.2168615681205157) circle (1.5pt);

\draw[thick,blue] (6.624111007838177, 3.1738427500474855) -- (6.07638284968393, 0.7121467007557514) -- (6.52346945948608, 0.15177297899840353) -- (9.228066257193067, 0.7794979281891886) -- (9.124402587309786, 1.7037920749506217) -- (8.881512072837118, 2.679392258307101) -- (7.603303001509191, 3.2878939726186216) -- cycle;

\draw[thick,red] (6.096365610731893, 1.3465016594940995) -- (6.202723314428924, 0.504622868580554) -- (7.387288836416293, 0.10803741976219167) -- (8.957273286214152, 0.471552478388929) -- (9.175567702081288, 1.9813374030108464) -- (9.069209998384256, 2.8232161939243916) -- (7.884644476396888, 3.219801642742754) -- (6.314660026599029, 2.856286584116017) -- cycle;

\draw[opacity=0.8,gray,thick] (6.624111007838177, 3.1738427500474855) -- (6.678151011964221, 2.94044948577554);
\draw[fill=black] (6.624111007838177, 3.1738427500474855) circle (1.5pt);
\draw[fill=black] (6.678151011964221, 2.94044948577554) circle (1.5pt);
\draw[opacity=0.8,gray,thick] (6.52346945948608, 0.15177297899840353) -- (6.597367046070866, 0.37249850860931893);
\draw[fill=black] (6.52346945948608, 0.15177297899840353) circle (1.5pt);
\draw[fill=black] (6.597367046070866, 0.37249850860931893) circle (1.5pt);
\draw[opacity=0.8,gray,thick] (7.3335028401356475, 0.3397781423344645) -- (7.387288836416293, 0.10803741976219167);
\draw[fill=black] (7.3335028401356475, 0.3397781423344645) circle (1.5pt);
\draw[fill=black] (7.387288836416293, 0.10803741976219167) circle (1.5pt);
\draw[opacity=0.8,gray,thick] (8.903295383005549, 0.7041200453673023) -- (8.957273286214152, 0.471552478388929);
\draw[fill=black] (8.903295383005549, 0.7041200453673023) circle (1.5pt);
\draw[fill=black] (8.957273286214152, 0.471552478388929) circle (1.5pt);
\draw[opacity=0.8,gray,thick] (8.881512072837118, 2.6793922583071006) -- (9.069209998384256, 2.8232161939243916);
\draw[fill=black] (8.881512072837118, 2.6793922583071006) circle (1.5pt);
\draw[fill=black] (9.069209998384256, 2.8232161939243916) circle (1.5pt);
\draw[opacity=0.8,gray,thick] (6.542190449019707, 2.8056609923690616) -- (6.314660026599029, 2.856286584116017);
\draw[fill=black] (6.542190449019707, 2.8056609923690616) circle (1.5pt);
\draw[fill=black] (6.314660026599029, 2.856286584116017) circle (1.5pt);

\draw[thick,blue] (0.3337576722656526, 4.2629821140984) -- (3.408071654497921, 4.747604427638509) -- (3.284720281839473, 6.035700624955344) -- (2.2773889491281114, 6.9069167074647035) -- (0.7552468047730609, 6.924969896989102) -- cycle;

\draw[thick,red] (0.3950402428022957, 6.481361871794583) -- (0.2677842231423603, 5.147474659968113) -- (0.5208321238412212, 4.482113940659169) -- (2.101430374993268, 4.250483702950762) -- (3.122283914739602, 4.665073592173341) -- (3.2495399343995377, 5.998960803999811) -- (2.9964920337006764, 6.664321523308754) -- (1.41589378254863, 6.895951761017162) -- cycle;

\draw[opacity=0.8,gray,thick] (0.3337576722656526, 4.2629821140984) -- (0.5208321238412211, 4.482113940659169);
\draw[fill=black] (0.3337576722656526, 4.2629821140984) circle (1.5pt);
\draw[fill=black] (0.5208321238412211, 4.482113940659169) circle (1.5pt);
\draw[opacity=0.8,gray,thick] (0.6779157857413833, 6.436572422358612) -- (0.3950402428022957, 6.481361871794583);
\draw[fill=black] (0.6779157857413833, 6.436572422358612) circle (1.5pt);
\draw[fill=black] (0.3950402428022957, 6.481361871794583) circle (1.5pt);
\draw[opacity=0.8,gray,thick] (2.0566479327920564, 4.534571505701716) -- (2.101430374993268, 4.250483702950762);
\draw[fill=black] (2.0566479327920564, 4.534571505701716) circle (1.5pt);
\draw[fill=black] (2.101430374993268, 4.250483702950762) circle (1.5pt);
\draw[opacity=0.8,gray,thick] (2.808803520186264, 6.447309321343283) -- (2.9964920337006764, 6.664321523308754);
\draw[fill=black] (2.808803520186264, 6.447309321343283) circle (1.5pt);
\draw[fill=black] (2.9964920337006764, 6.664321523308754) circle (1.5pt);

\draw[thick,blue] (6.29599744875811, 6.609093971429123) -- (6.476929755979502, 4.567862591568923) -- (8.257759275442012, 4.841483785148059) -- (8.810646390238388, 5.128874813439007) -- (9.38772162066735, 5.819559598706189) -- (8.78551981168817, 6.966824990266439) -- cycle;

\draw[thick,red] (9.224050126949415, 5.843336399564493) -- (8.893615330544378, 7.093904665258737) -- (7.925173294280466, 7.00763782306017) -- (6.536590019417608, 6.762653907653567) -- (6.192737992029044, 5.896854628973563) -- (6.523172788434081, 4.646286363279319) -- (7.491614824697992, 4.7325532054778865) -- (8.88019809956085, 4.97753712088449) -- cycle;

\draw[opacity=0.8,gray,thick] (6.29599744875811, 6.609093971429123) -- (6.451134296877382, 6.547481400176104);
\draw[fill=black] (6.29599744875811, 6.609093971429123) circle (1.5pt);
\draw[fill=black] (6.451134296877382, 6.547481400176104) circle (1.5pt);
\draw[opacity=0.8,gray,thick] (9.38772162066735, 5.819559598706189) -- (9.224050126949415, 5.843336399564493);
\draw[fill=black] (9.38772162066735, 5.819559598706189) circle (1.5pt);
\draw[fill=black] (9.224050126949415, 5.843336399564493) circle (1.5pt);
\draw[opacity=0.8,gray,thick] (8.78551981168817, 6.96682499026644) -- (8.893615330544378, 7.093904665258737);
\draw[fill=black] (8.78551981168817, 6.96682499026644) circle (1.5pt);
\draw[fill=black] (8.893615330544378, 7.093904665258737) circle (1.5pt);
\draw[opacity=0.8,gray,thick] (7.948324415257536, 6.846524499461226) -- (7.925173294280466, 7.00763782306017);
\draw[fill=black] (7.948324415257536, 6.846524499461226) circle (1.5pt);
\draw[fill=black] (7.925173294280466, 7.00763782306017) circle (1.5pt);
\draw[opacity=0.8,gray,thick] (6.35783237217902, 5.911488397298184) -- (6.192737992029044, 5.896854628973563);
\draw[fill=black] (6.35783237217902, 5.911488397298184) circle (1.5pt);
\draw[fill=black] (6.192737992029044, 5.896854628973563) circle (1.5pt);
\draw[opacity=0.8,gray,thick] (8.803471331179377, 5.12514521376878) -- (8.88019809956085, 4.97753712088449);
\draw[fill=black] (8.803471331179377, 5.12514521376878) circle (1.5pt);
\draw[fill=black] (8.88019809956085, 4.97753712088449) circle (1.5pt);

\draw[thick,blue] (0.6653939437938469, 10.823455405083475) -- (0.8000399810077479, 8.013022849381413) -- (2.550473941095362, 8.65327303331315) -- (3.4820064767164483, 10.01325283472592) -- (2.4402384773543133, 11.312686814314135) -- cycle;

\draw[thick,red] (0.4833906287644487, 9.529562442071004) -- (1.0139087151772253, 8.132145035077285) -- (2.025077994878358, 8.218527746196642) -- (3.0730944468160426, 8.982980114560007) -- (3.2432611819266874, 9.966909820565267) -- (2.7127430955139107, 11.364327227558984) -- (1.7015738158127787, 11.277944516439629) -- (0.6535573638750931, 10.513492148076262) -- cycle;

\draw[opacity=0.8,gray,thick] (0.6653939437938469, 10.823455405083475) -- (0.8088591378584612, 10.626773584367758);
\draw[fill=black] (0.6653939437938469, 10.823455405083475) circle (1.5pt);
\draw[fill=black] (0.8088591378584612, 10.626773584367758) circle (1.5pt);
\draw[opacity=0.8,gray,thick] (0.8000399810077479, 8.013022849381413) -- (1.0139087151772253, 8.132145035077285);
\draw[fill=black] (0.8000399810077479, 8.013022849381413) circle (1.5pt);
\draw[fill=black] (1.0139087151772253, 8.132145035077285) circle (1.5pt);
\draw[opacity=0.8,gray,thick] (3.4820064767164483, 10.01325283472592) -- (3.2432611819266883, 9.966909820565267);
\draw[fill=black] (3.4820064767164483, 10.01325283472592) circle (1.5pt);
\draw[fill=black] (3.2432611819266883, 9.966909820565267) circle (1.5pt);
\draw[opacity=0.8,gray,thick] (0.7268247770019693, 9.541225217367362) -- (0.4833906287644487, 9.529562442071004);
\draw[fill=black] (0.7268247770019693, 9.541225217367362) circle (1.5pt);
\draw[fill=black] (0.4833906287644487, 9.529562442071004) circle (1.5pt);
\draw[opacity=0.8,gray,thick] (2.871086538344019, 9.121347566046401) -- (3.0730944468160426, 8.982980114560007);
\draw[fill=black] (2.871086538344019, 9.121347566046401) circle (1.5pt);
\draw[fill=black] (3.0730944468160426, 8.982980114560007) circle (1.5pt);
\draw[opacity=0.8,gray,thick] (2.5216564366830316, 11.21113132118587) -- (2.7127430955139107, 11.364327227558984);
\draw[fill=black] (2.5216564366830316, 11.21113132118587) circle (1.5pt);
\draw[fill=black] (2.7127430955139107, 11.364327227558984) circle (1.5pt);

\draw[thick,blue] (6.219995247432211, 9.882188364844733) -- (6.08847813255473, 9.564055723005996) -- (6.644142779724204, 8.531049609187075) -- (7.627940724114183, 8.029882448718906) -- (9.223082582501085, 8.366438799405184) -- (9.361156361829577, 8.432633714390057) -- (7.7089754266220005, 10.609721610863089) -- (7.4042771398807234, 10.875332714314474) -- cycle;

\draw[thick,red] (7.755446070291139, 7.848761475408373) -- (9.189198446608383, 8.54417168762233) -- (8.944606476557553, 9.271617549413618) -- (8.24747083138991, 10.235361113714369) -- (7.481617875614431, 10.688002170402225) -- (6.047865499297187, 9.992591958188267) -- (6.292457469348016, 9.265146096396979) -- (6.98959311451566, 8.301402532096228) -- cycle;

\draw[opacity=0.8,gray,thick] (9.361156361829577, 8.432633714390057) -- (9.189198446608383, 8.54417168762233);
\draw[fill=black] (9.361156361829577, 8.432633714390057) circle (1.5pt);
\draw[fill=black] (9.189198446608383, 8.54417168762233) circle (1.5pt);
\draw[opacity=0.8,gray,thick] (7.4042771398807234, 10.875332714314474) -- (7.481617875614431, 10.688002170402225);
\draw[fill=black] (7.4042771398807234, 10.875332714314474) circle (1.5pt);
\draw[fill=black] (7.481617875614431, 10.688002170402225) circle (1.5pt);
\draw[opacity=0.8,gray,thick] (7.713426177121031, 8.04791888350336) -- (7.755446070291139, 7.848761475408373);
\draw[fill=black] (7.713426177121031, 8.04791888350336) circle (1.5pt);
\draw[fill=black] (7.755446070291139, 7.848761475408373) circle (1.5pt);
\draw[opacity=0.8,gray,thick] (8.086044562320737, 10.1128555388913) -- (8.24747083138991, 10.235361113714369);
\draw[fill=black] (8.086044562320737, 10.1128555388913) circle (1.5pt);
\draw[fill=black] (8.24747083138991, 10.235361113714369) circle (1.5pt);
\draw[opacity=0.8,gray,thick] (6.219995247432211, 9.882188364844733) -- (6.047865499297187, 9.992591958188267);
\draw[fill=black] (6.219995247432211, 9.882188364844733) circle (1.5pt);
\draw[fill=black] (6.047865499297187, 9.992591958188267) circle (1.5pt);

\draw[thick,blue] (3.356628113220527, 15.027103136609968) -- (1.0220082433486612, 15.327207409429565) -- (0.4977027893868573, 15.303887234870587) -- (0.3927289719838458, 12.458379689431984) -- (0.893628566665871, 12.298345426331293) -- (2.974133324735726, 13.227499813983357) -- cycle;

\draw[thick,red] (3.3269521375012436, 14.071865812203104) -- (2.964963433229304, 15.368384854418721) -- (0.6378177402909669, 15.052740240627049) -- (0.34370200438682424, 14.718451745321182) -- (0.1541892192482841, 13.834079414024517) -- (0.5161779235202245, 12.5375603718089) -- (2.843323616458561, 12.85320498560057) -- (3.137439352362704, 13.18749348090644) -- cycle;

\draw[opacity=0.8,gray,thick] (3.356628113220527, 15.027103136609968) -- (3.0816822676487536, 14.950338131899048);
\draw[fill=black] (3.356628113220527, 15.027103136609968) circle (1.5pt);
\draw[fill=black] (3.0816822676487536, 14.950338131899048) circle (1.5pt);
\draw[opacity=0.8,gray,thick] (0.4977027893868573, 15.303887234870587) -- (0.637817740290967, 15.052740240627049);
\draw[fill=black] (0.4977027893868573, 15.303887234870587) circle (1.5pt);
\draw[fill=black] (0.637817740290967, 15.052740240627049) circle (1.5pt);
\draw[opacity=0.8,gray,thick] (0.893628566665871, 12.298345426331293) -- (0.8549499933672067, 12.583510042890834);
\draw[fill=black] (0.893628566665871, 12.298345426331293) circle (1.5pt);
\draw[fill=black] (0.8549499933672067, 12.583510042890834) circle (1.5pt);
\draw[opacity=0.8,gray,thick] (2.928173032566042, 15.082178998027498) -- (2.964963433229304, 15.368384854418721);
\draw[fill=black] (2.928173032566042, 15.082178998027498) circle (1.5pt);
\draw[fill=black] (2.964963433229304, 15.368384854418721) circle (1.5pt);
\draw[opacity=0.8,gray,thick] (0.4430868310106908, 13.823421672071266) -- (0.1541892192482841, 13.834079414024517);
\draw[fill=black] (0.4430868310106908, 13.823421672071266) circle (1.5pt);
\draw[fill=black] (0.1541892192482841, 13.834079414024517) circle (1.5pt);
\draw[opacity=0.8,gray,thick] (2.7257115800893708, 13.116554547873761) -- (2.843323616458561, 12.85320498560057);
\draw[fill=black] (2.7257115800893708, 13.116554547873761) circle (1.5pt);
\draw[fill=black] (2.843323616458561, 12.85320498560057) circle (1.5pt);

\draw[thick,blue] (6.644297201211346, 15.092941303612598) -- (6.22605024884862, 12.819547297754957) -- (8.487155136357643, 12.24819589393315) -- (8.987329346553357, 12.280501298172414) -- (8.979061874246227, 13.595196090743753) -- (7.35909939899043, 15.00378560304296) -- cycle;

\draw[thick,red] (6.181642850240471, 13.814453014340767) -- (6.433415027456646, 12.902285940533169) -- (8.13816092900468, 12.105407936733412) -- (8.754284580316025, 12.222139114292691) -- (8.939606849931721, 13.230902851942709) -- (8.687834672715546, 14.143069925750307) -- (6.983088771167512, 14.939947929550064) -- (6.366965119856167, 14.823216751990785) -- cycle;

\draw[opacity=0.8,gray,thick] (6.22605024884862, 12.819547297754957) -- (6.433415027456646, 12.902285940533169);
\draw[fill=black] (6.22605024884862, 12.819547297754957) circle (1.5pt);
\draw[fill=black] (6.433415027456646, 12.902285940533169) circle (1.5pt);
\draw[opacity=0.8,gray,thick] (8.987329346553357, 12.280501298172414) -- (8.772264896514082, 12.320011275095036);
\draw[fill=black] (8.987329346553357, 12.280501298172414) circle (1.5pt);
\draw[fill=black] (8.772264896514082, 12.320011275095036) circle (1.5pt);
\draw[opacity=0.8,gray,thick] (6.401641492493177, 13.773978832826675) -- (6.181642850240471, 13.814453014340767);
\draw[fill=black] (6.401641492493177, 13.773978832826675) circle (1.5pt);
\draw[fill=black] (6.181642850240471, 13.814453014340767) circle (1.5pt);
\draw[opacity=0.8,gray,thick] (8.193022140644487, 12.32251942529138) -- (8.13816092900468, 12.105407936733412);
\draw[fill=black] (8.193022140644487, 12.32251942529138) circle (1.5pt);
\draw[fill=black] (8.13816092900468, 12.105407936733412) circle (1.5pt);
\draw[opacity=0.8,gray,thick] (8.541939968608512, 13.975282260196726) -- (8.687834672715546, 14.143069925750307);
\draw[fill=black] (8.541939968608512, 13.975282260196726) circle (1.5pt);
\draw[fill=black] (8.687834672715546, 14.143069925750307) circle (1.5pt);
\draw[opacity=0.8,gray,thick] (6.587219819449563, 14.782695462407393) -- (6.366965119856167, 14.823216751990785);
\draw[fill=black] (6.587219819449563, 14.782695462407393) circle (1.5pt);
\draw[fill=black] (6.366965119856167, 14.823216751990785) circle (1.5pt);

\end{tikzpicture}
\end{center}
\caption{Pairs of polytopes $P,Z$ where $P$ (blue) is randomly generated and $Z$ (red) is a rank $4$ local minimum of $d_P$ found using Algorithm \ref{alg:subgradient}. The black points are the pairs $(p,q) \in P \times Z$ that achieve the Hausdorff distance $d_P(Z)$.}
\label{fig:eight-examples}
\end{figure}
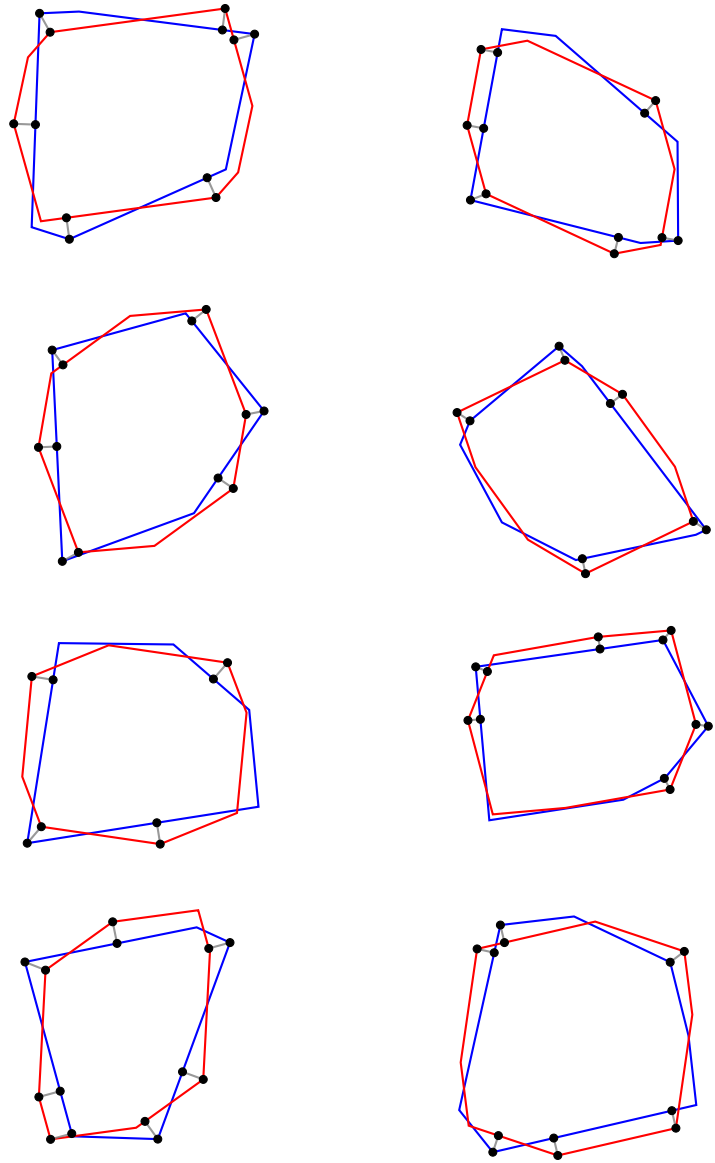


\ \\
{\bf Acknowledgements}: I'd like to acknowledge Joe Kileel who was a helpful resource while drafting this work and provided great suggestions, as well as Diane Maclagan who was a careful and thoughtful reviewer. I would also like to credit the authors of the OSQP solver and the \texttt{qpsolvers} python package, which are used in the software implementation of this work \cite{OSQP,qpsolvers}. Finally, I'd like to thank the University of Texas at Austin Mathematics Department, which originally accepted this as part of my Ph.D. dissertation in August 2023. 
\clearpage

\bibliographystyle{unsrt}
\bibliography{zonopt}

\begin{thebibliography}{10}

\bibitem{McM70}
Peter McMullen.
\newblock {Polytopes with Centrally Symmetric Faces}.
\newblock {\em Israel Journal of Mathematics}, 8:194--196, 1970.

\bibitem{Shor85}
Naum Shor.
\newblock {\em {Minimization Methods for Non-Differentiable Functions}}.
\newblock Springer Berlin, Heidelberg, 1985.

\bibitem{MSRM22}
Panagiotis Misiakos, Georgios Smyrnis, George Retsinas, and Petros Maragos.
\newblock {Neural Network Approximation based on Hausdorff distance of Tropical
  Zonotopes}.
\newblock In {\em International Conference on Learning Representations}, 2022.

\bibitem{GNZ03}
Leonidas~J. Guibas, An~Thanh Nguyen, and Li~Zhang.
\newblock {Zonotopes as bounding volumes}.
\newblock In {\em ACM-SIAM Symposium on Discrete Algorithms}, 2003.

\bibitem{BAC06}
J.M. Bravo, T.~Alamo, and E.F. Camacho.
\newblock {Bounded error identification of systems with time-varying
  parameters}.
\newblock {\em IEEE Transactions on Automatic Control}, 51(7), 2006.

\bibitem{COM15}
Christophe Combastel.
\newblock {Zonotopes and Kalman observers: Gain optimality under distinct
  uncertainty paradigms and robust convergence}.
\newblock {\em Automatica}, 55:265--273, 2015.

\bibitem{Ku98}
W.~K{\"u}hn.
\newblock {Rigorously computed orbits of dynamical systems without the wrapping
  effect}.
\newblock {\em Computing}, 61(1), 1998.

\bibitem{ASB08}
Matthias Althoff, Olaf Stursberg, and Martin Buss.
\newblock {Reachability analysis of nonlinear systems with uncertain parameters
  using conservative linearization}.
\newblock In {\em {2008 47th IEEE Conference on Decision and Control}}, 2008.

\bibitem{Gir05}
Antoine Girard.
\newblock Reachability of uncertain linear systems using zonotopes.
\newblock In Manfred Morari and Lothar Thiele, editors, {\em Hybrid Systems:
  Computation and Control}, pages 291--305. Springer Berlin Heidelberg, 2005.

\bibitem{YS18}
Xuejiao Yang and Joseph~K. Scott.
\newblock {A comparison of zonotope order reduction techniques}.
\newblock {\em Automatica}, 95:378--384, 2018.

\bibitem{RW98}
R.~Tyrrell Rockafellar and Roger J.~B. Wets.
\newblock {\em {Variational Analysis}}.
\newblock Springer, 1998.

\bibitem{GO97}
Jacob~E. Goodman and Joseph O'Rourke, editors.
\newblock {\em {Handbook of Discrete and Computational Geometry}}.
\newblock CRC Press, Inc., USA, 1997.

\bibitem{McM71}
Peter McMullen.
\newblock {On Zonotopes}.
\newblock {\em Transactions of the American Mathematical Society}, 159, 1971.

\bibitem{Kon14}
Stefan K{\"o}nig.
\newblock {Computational Aspects of the Hausdorff Distance in Unbounded
  Dimension}.
\newblock {\em arXiv}, 2014.

\bibitem{Bur14}
James~V. Burke.
\newblock {Nonlinear Optimization}.
\newblock {\em Lecture Notes, The University of Washington}, 2014.

\bibitem{LMK20}
Jiajin Li, Anthony Man-Cho So, and Wing-Kin Ma.
\newblock {Understanding Notions of Stationarity in Non-Smooth Optimization}.
\newblock {\em arXiv}, 2020.

\bibitem{Clarke90}
Frank~H. Clarke.
\newblock {\em Optimization and Nonsmooth Analysis}.
\newblock Society for Industrial and Applied Mathematics, 1990.

\bibitem{BT21}
Adil~M. Bagirov, Sona Taheri, Jaisa Joki, Napsu Karmitsa, and Marko~M.
  M{\"a}kel{\"a}.
\newblock {Aggregate subgradient method for nonsmooth DC optimization}.
\newblock {\em Optmization Letters}, 15(1), 2021.

\bibitem{Kiw86}
Krzysztof~C. Kiwiel.
\newblock {An Aggregate Subgradient Method for Nonsmooth and Nonconvex
  Minimization}.
\newblock {\em Journal of Computational and Applied Mathematics},
  14(3):391--400, 1986.

\bibitem{BJ13}
A.~M. Bagirov, L.~Jin, N.~Karmitsa, A.~Al Nuaimat, and N.~Sultanova.
\newblock {Subgradient Method for Nonconvex Nonsmooth Optimization}.
\newblock {\em Journal of Optimization Theory and Applications},
  157(2):416--435, 2013.

\bibitem{BL04}
Stephen Boyd and Lieven Vandenberghe.
\newblock {\em Convex Optimization}.
\newblock Cambridge University Press, 2004.

\bibitem{OSQP}
B.~Stellato, G.~Banjac, P.~Goulart, A.~Bemporad, and S.~Boyd.
\newblock {OSQP: an operator splitting solver for quadratic programs}.
\newblock {\em Mathematical Programming Computation}, 12(4):637--672, 2020.

\bibitem{qpsolvers}
St\'{e}phane Caron, Daniel Arnstr\"{o}m, Suraj Bonagiri, Antoine Dechaume,
  Nikolai Flowers, Adam Heins, Takuma Ishikawa, Dustin Kenefake, Giacomo
  Mazzamuto, Donato Meoli, Brendan O'Donoghue, Adam~A. Oppenheimer, Abhishek
  Pandala, Juan~Jos\'{e} Quiroz Oma\~{n}a, Nikitas Rontsis, Paarth Shah, Samuel
  St-Jean, Nicola Vitucci, Soeren Wolfers, Fengyu Yang, @bdelhaisse,
  @MeindertHH, @rimaddo, @urob, and @shaoanlu.
\newblock {qpsolvers: Quadratic Programming Solvers in Python}, March 2024.

\end{thebibliography}

\ \\

\begin{center}
\small{\textsc{Austin, TX 78733, USA}}

\small{{\it Email address:} \url{gdavtor@gmail.com}}
\end{center}

\end{document}